%% file: paper_1.tex
\documentclass[12]{amsart} 
\usepackage{graphicx,amsmath,amsthm,amssymb,latexsym,amsfonts,color,float}
\usepackage[bookmarksnumbered, colorlinks, plainpages]{hyperref}
\usepackage{algorithm2e}
\usepackage{booktabs}
\usepackage{hhline,booktabs}
\usepackage{multirow}
\usepackage{soul}
\usepackage{footnotehyper}
\usepackage{tikz}
\usepackage[mathscr]{euscript}
\usepackage{pifont}

\input{macros.tex}

\graphicspath{{./FIGS/}}

\begin{document}
	
	\thispagestyle{empty}	
	\title{On the tubular eigenvalues of third-order tensors}
	
	\author[F. P. A. Beik and Y. Saad]{Fatemeh P. A. Beik$^\S$ and Yousef Saad$^{\ddag}$} 
	
	\maketitle
	
	{\centering
		\noindent $^\S${\small{Department of Mathematics, Vali-e-Asr University of Rafsanjan, P.O. Box 518, Rafsanjan, Iran}}\\
		\texttt{e-mails: f.beik@vru.ac.ir; beik.fatemeh@gmail.com}
		
		\noindent $^\ddag${\small{Computer Science \& Engineering, University of Minnesota, Twin Cities, USA}}\\
		\texttt{e-mail:  saad@umn.edu}\\[0.2cm]
		\date{\today}
	}
	
	\begin{abstract}
		This paper introduces the notion of tubular eigenvalues of third-order
		tensors with respect to T-products of tensors and analyzes their properties. A focus of
		the paper is to discuss relations between tubular eigenvalues and two
		alternative definitions of eigenvalue for third-order tensors that are known
		in the literature, namely eigentuple and T-eigenvalue. In
		addition, it establishes a few results on tubular spectra of tensors which can
		be exploited to analyze the convergence of tubular versions of iterative
		methods for solving tensor equations.  \bigskip
		
		\noindent \textit{Keywords}: Multilinear algebra, tensor, T-product,  T-eigenvalue, eigentuple, Iterative method.\\
		\noindent \textit{2010 AMS Subject Classification}: 15A69, 65F10.
	\end{abstract}

	\section{Introduction and preliminaries}\label{sec:intro}
	Throughout this paper, we consider third-order tensors 
	denoted by calligraphic letters, e.g., $\cX$, which are members of
	$\C^{n \times m \times p}$. Associated with such tensors, we denote by
	$ X_k\; \; {\rm for } \quad k=1,2,\ldots,p$ the matrix of order
	$n \times m $ obtained by fixing the last index to $k$. This matrix is
	commonly referred to as a \emph{frontal slice.}  The T-product was first
	introduced in \cite{Kilmer2011} as a tensor-tensor multiplication.
	When working with   high-dimensional data,   this
	product was found to be a particularly useful tool  when compared with matricization, see
	\cite{Hao,Khaleel,Kilmer2013,Kilmer2021,Reichel,Zeng} and the references
	therein.
	
	Generalizing the concept of eigenvalues  from matrices to tensors  has also
	been  considered using  T-products. For  instance, Braman  \cite{Braman}
	defined the eigenvalue of $n\times n \times n$ tensor as a vector rather
	than  a  scalar.  The  authors of  \cite{Qi}   defined the  concepts  of
	\emph{eigentuples}  and \emph{eigenmatrices}  of  $n\times  n \times  p$
	tensors as an extension  of the idea used in  \cite{Braman}. Recently, a few
	publications were   devoted  to the  study  of  eigenvalues  of
	$n\times n  \times p$ tensors as  scalars \cite{Liu,Miao,Miao2021} which
	are  called T-eigenvalues  of tensors.  In this  paper, we  consider the
	eigenvalue of an $n\times n \times p$  tensor as an $1\times 1 \times p$
	tensor called a  \emph{tubular} eigenvalue. The properties  of the above
	mentioned types of eigenvalues are  analyzed in detail. In addition, the
	paper  will aim  to present  a number  of relations  between tubular
	eigenvalues and  T-eigenvalues as well  as eigentuples. It  should be noted
	that there has been an increased interest in applying iterative methods
	for  solving  tensor equations  with  respect  to T-product,    see, e.g., 
	\cite{El2021tensor,ElIchi,Ma2022randomized,Reichel2022tensor,Reichel,Reichel2022weighted}.
	Therefore,
	we  also establish a few theoretical  results that  can be  utilized to
	analyze the convergence  of iterative methods with  respect to tubular
	eigenvalues and   point out the benefits of working with tubular
	eigenvalues instead of T-eigenvalues.
	
	The remainder of this paper is organized as follows. The  end of this
	section  presents the notations, definitions, and a few 
	preliminaries  used throughout the paper. The definition and
	properties of the tubular eigenvalues are presented in Section
	\ref{sec3}.  Section \ref{sec4}, introduces the notion of tubular
	spectral radius and establishes some convergence     results  of
	(non)-stationary iterative methods. The paper ends 
	with a few  concluding remarks in Section \ref{sec5}.

	\subsection{Definitions, notations and properties} The scalar Frobenius
	norm of a tensor $\mathscr{X}\in \mathbb{C}^{n\times m \times p}$ is
	given by (see \cite{Kolda})
	\[
	\left\| \mathscr{X} \right\|^2_F = {\sum\limits_{i_1  = 1}^{n} {\sum\limits_{i_2  = 1}^{m} { \sum\limits_{i_3  = 1}^{p} {|\mathscr{X} _{i_1 i_2i_3}|^2 } } } }^{} .
	\]
	The scalar Frobenius inner product of two tensors
	$\mathscr{X}, \mathscr{Y} \in \mathbb{C}^{n\times m \times p}$ of the same shape is defined by
	\begin{equation*}
		\left\langle {\mathscr{X}, \mathscr{Y}} \right\rangle_F  = \sum\limits_{i_1  = 1}^{n} \sum\limits_{i_2  = 1}^{m } \sum\limits_{i_3  = 1}^{p} {\rm conj}({\mathscr{X}_{i_1 i_2 i_3 }})\mathscr{Y}_{i_1 i_2 i_3 }
	\end{equation*}
	where ${\rm conj}({\mathscr{X}_{i_1 i_2 i_3 }})$ stands for the complex
	conjugate of $\mathscr{X}_{i_1 i_2 i_3 } $
	
	A $1 \times 1 \times p $ tensor $\mathscr{V}$ is called a \emph{tubular tensor}
	of length $p$.  To a vector $v \ \in \C^p$ with components $v_1,v_2,\ldots,v_p$,
	we can associate a tubular tensor $\cV$, in a canonical way by setting
	$\cV_{11j} = v_j$ for $j=1,2,\ldots,p$. Note that we can always associate a vector $v$ to
	a tubular tensor $\cV$ by defining the
	$j$th entry of $v$  to  be $\cV_{11j}$ for $j=1,2,\ldots,n$. Throughout the paper, the tubular tensor $\cV$ is denoted by $[v]$. The tubular tensor  $\cV $ may be the result of 
	an operation with non-tubular tensors and the vector 
	$v$ is implicit but not required in our notation.
	With the tubular tensor $[v]$ of length $p$, we associate the
	circulant matrix:
	\eq{eq:circMat} \crc ([v]) :=\left( {\begin{array}{*{20}{c}}
			{{v_1}}&{{v_{{p}}}}&{{v_{{{p - 1}}}}}& \ldots &{{v_2}}\\
			{{v_2}}&{{v_1}}&{{v_{{p}}}}& \ldots &{{v_3}}\\
			\vdots & \ddots & \ddots & \ddots & \vdots \\
			{{v_{{p}}}}&{{v_{{{p-1}}}}}&
			\ddots
			&{{v_2}}&{{v_1}}
	\end{array}} \right)
	\en 
	where $v_1,v_2,\ldots,v_p$ are the frontal slices of  $[v]$. It is well-known that
	$\crc([v])$ is diagonalized by a discrete Fourier transform (DFT), i.e.,
	$
	\crc([v]) = F_p D F^H_p
	$
	where the DFT matrix $F_p$ is defined by (see \cite{Chan}),
	\eq{eq:FpMat}
	{F}_p = \frac{1}{{\sqrt {{p}} }}\left( {\begin{array}{*{20}{c}}
			1&1&1&1& \cdots &1\\
			1&{{\omega ^{}}}&{{\omega ^2}}&{{\omega ^3}}& \cdots &{{\omega ^{{p} - 1}}}\\
			1&{{\omega ^2}}&{{\omega ^4}}&{{\omega ^6}}& \cdots &{{\omega ^{2({p} - 1)}}}\\
			1&{{\omega ^3}}&{{\omega ^6}}&{{\omega ^9}}& \cdots &{{\omega ^{3({p} - 1)}}}\\
			\vdots & \vdots & \vdots & \vdots & \ddots & \vdots \\
			1&{{\omega ^{p - 1}}}&{{\omega ^{2(p - 1)}}}&{{\omega ^{3(p - 1)}}}& \cdots &{{\omega ^{(p - 1)(p- 1)}}}
	\end{array}} \right),
	\en
	\noindent here $\omega =e^{-2\pi {\bf i} /p}$ is the primitive $p$--th root of
	unity in which ${\bf i}=\sqrt{-1}$ and $F^H$ denotes the conjugate transpose of
	$F$. Following \cite{Kilmer2013}, this is generalized to 3-way tensors as
	follows. For a tensor $\cX \ \in \ \C^{n \times m \times p}$ we denote by
	$X_1, X_2, \cdots, X_p$ the $p$ frontal slices of $\cX$ and associate to $\cX$
	the block-circulant matrix:
	\[\bcrc(\mathscr {X})=\left( {\begin{array}{*{20}{c}}
			{{X_1}}&{{X_{{p}}}}&{{X_{{p{- 1}}}}}& \ldots &{{X_2}}\\
			{{X_2}}&{{X_1}}&{{X_{{p}}}}& \ldots &{{X_3}}\\
			\vdots & \ddots & \ddots & \ddots & \vdots \\
			{{X_{{p}}}}&{{X_{{p{- 1}}}}}& \ddots &{{X_2}}&{{X_1}}
	\end{array}} \right)\in \mathbb{R}^{{np\times mp}} . \]
	It is pointed out in \cite[Sec. 2.5]{Kilmer2013},
	that block-circulant matrices can be block-diagonalized as follows:
	\begin{equation}\label{bld}
		\bcrc(\mathscr{X})=(F_p\otimes I_n) {\rm blockdiag}(\tilde{X}_1,\tilde{X}_2,\ldots,\tilde{X}_p)(F_p^H\otimes I_m)
	\end{equation}
	where the {$n \times m$} matrices $\tilde{X}_1,\tilde{X}_2,\ldots,\tilde{X}_p$  will be complex unless certain symmetry conditions hold.
	
	We also  denote the unfold operation \cite{Kilmer2013}
	by $\ufld$  and and its inverse by $\fld$:
	\[
	\ufld(\mathscr {X} ) = \begin{pmatrix}
		X_{1}   \\
		X_{2}   \\
		\vdots \\
		X_{p}\end{pmatrix} \in \mathbb{R}^{np\times m} ,
	\qquad \fld(\ufld(\mathscr {X}) ) =  \mathscr {X}
	\]     
	With this the article \cite{Kilmer2013} defines the T-product of 3-way tensors as follows: 
	
	\begin{defn}
		The \textbf{T-product}, denoted by $\ast$, between  two tensors
		$\mathscr {X} \in \mathbb{R}^{n\times m\times p} $ and $\mathscr {Y} \in \mathbb{R}^{m\times \ell \times p} $ is the following  ${n \times \ell \times p}$ tensor~:
		$$\mathscr {X}\ast\mathscr {Y}=\fld(\bcrc(\mathscr {X})\ufld(\mathscr {Y}) ) . $$
	\end{defn}
	
	Note that the T-product was originally defined in \cite[Definition 3.1]{Kilmer2011}.
	The T-product of two tensors is not commutative in general
	\cite[Example 2.6]{Kilmer2013}. However, it is commutative for tubular tensors,
	an immediate consequence of the property that two circulant matrices of the same
	dimension are diagonalizable by the same matrix $F_p$ - as seen above. It can be
	verified that the T-product is associative, i.e.,
	$\mathscr{A}\ast(\mathscr{B} \ast \mathscr{C})=(\mathscr{A}\ast\mathscr{B}) \ast
	\mathscr{C}$, see \cite[Lemma 3.3]{Kilmer2011}.

	Throughout this paper, the zero tensor of size $n\times 1 \times p$ is denoted
	by $\mathscr{O}$. The $n \times n\times p$ identity tensor $\mathscr{I}_{ nnp}$
	is the tensor whose first frontal slice is the $ n \times n$ identity matrix,
	and the other frontal slices are all zeros, i.e.,
	$$\ufld({\mathscr{I}_{ nnp}})  = \begin{pmatrix}
		I_{n}   \\
		0 _{n} \\
		\vdots \\
		0_{n}\end{pmatrix}$$
	where $I_{n}$ and $0 _{n}$ are the $n \times n$ identity and zero matrices, respectively. In the special case when $n=1$, the identity tensor becomes the tubular tensor  $[e_1]$, where $e_1  = (1,0,0\ldots,0)^T $ is the first column of the $p \times p$ identity matrix.

	In addition, we  use the following basic definitions:
	
	\begin{enumerate}
		\item An $n\times n \times p$ tensor $\mathscr{A}$ is non-singular, if there exists a tensor $\mathscr{B}$ of order  $n \times n \times p$  such that
		$$\mathscr{A}\ast \mathscr{B}=\mathscr{I}_{ n  n  p} \qquad \text{and}\qquad \mathscr{B}\ast \mathscr{A}=\mathscr{I}_{ n  n  p}$$
		It can be seen that $\mathscr{A}$ is non-singular if and only if   $\bcrc(\mathscr{A})$ is non-singular; see \cite{Miao}.
		\item If $\mathscr{A} $ is a real-valued $n \times m \times p$ tensor, then
		$\mathscr{A}^{T}$ is the $m \times n \times p$ tensor obtained by transposing
		each of the front-back frontal slices and then reversing the order of the
		transposed frontal slices 2 through $p$.  For a complex-valued tensor
		$\mathscr{A}$, the conjugate transpose of $\mathscr{A}$ is defined by
		$\mathscr{A}^H=(\bar{ \mathscr{A}})^T$.  An $n\times n \times p$ real
		(complex) valued tensor $\mathscr{A}$ is called symmetric $( {Hermitian})$, if
		$\mathscr{A}=\mathscr{A}^{T}$ $(\mathscr{A}=\mathscr{A}^{H})$.
		\item A tensor $\mathscr{Q} \in \mathbb{C}^{n\times n \times p}$ is called unitary if it satisfies:
		\[
		\mathscr{Q}^H \ast \mathscr{Q} =\mathscr{Q} \ast \mathscr{Q}^H=\mathscr{I}_{ nnp}.
		\]
		A real-valued tensor $\mathscr{Q}$ that satisfies the above relation is called
		an orthogonal tensor \cite[Definition 3.18]{Kilmer2011}.
	\end{enumerate}
	
	It should be noted that $\bcrc(\mathscr {A})^T=\bcrc(\mathscr {A}^T)$ and that
	this may serve as an alternative definition for $\cA^T$. Therefore, $\cA$ is
	Hermitian (complex), resp. symmetric (real), iff $\bcrc(\cA)$ is Hermitian
	(complex), resp. symmetric (real).

	\subsection{Basic concepts}\label{sec:basics}
	The positive definiteness of tensors in terms of T-product  defined
	in \cite{Beik} can be regarded as a natural extension of the
	same concept for matrices.
	
	\begin{defn}\label{def:spd}
		The tensor $\mathscr{A}\in \mathbb{C}^{n\times n \times p_{}}$
		is said to be positive (semi) definite if $$(\mathscr{X}^H \ast \mathscr{A} \ast \mathscr{X})_{::1} > (\ge ) 0,$$
		for all nonzero tensors $\mathscr{X}\in \mathbb{C}^{n\times 1 \times p}$.
	\end{defn}
	
	When $\mathscr{A}$ is symmetric/Hermitian positive (semi)definite, we write
	$\mathscr{A}\succ 0$ ($\mathscr{A}\succcurlyeq 0$). Furthermore, for two given
	symmetric/Hermitian tensors $\mathscr{A}$ and $\mathscr{B}$, the notation $\mathscr{A}\succ \mathscr{B}$
	($\mathscr{A}\succcurlyeq \mathscr{B}$) means that
	$\mathscr{A}-\mathscr{B}\succ 0$ ($\mathscr{A}-\mathscr{B}\succcurlyeq 0$).
	The symbols $\prec,  \preccurlyeq$ are defined analogously.
	
	\begin{rem}\label{rem1.5}
		A useful observation is that when $n = 1$, i.e., when $\cA$ and $\cX$ are both tubular tensors, which we now denote by $[a]$ and $[x]$, respectively, then the product $[x]^H \ast [a] \ast [x]$  becomes
		\begin{align} 
			[x]^H \ast [a] \ast [x]
			&= \fld \left[ \crc([x])^H \ufld
			\left( \fld\left(\crc([a]) \ufld([x]) \right) \right) \right]  \nonumber \\
			& = \fld [\crc([x])^H \crc([a]) x ] \nonumber \\
			&=  [y] \label{eq:bil0} 
		\end{align} 
		where the vector $y\in \C^p$ is defined by $y=\crc([x])^H \ \crc([a]) x$.
		Note that the first row of $\crc([x])^H$ is
		$x^H$ and so the 1st  entry of the  tubular in \eqref{eq:bil0}  is
		$x^H \crc(a) x$, i.e., we can write:
		\eq{eq:bil} 
		([x]^H \ast [a] \ast [x])_{::1} = {x}^H \crc([a]) {x} . 
		\en
	\end{rem}

	We saw earlier that for a given vector $v$, the matrix $D = F^H_p\crc([v])F_p $ is diagonal. In other words, the matrix $\crc([v])$ is
	unitarily similar to a diagonal matrix $D$, and   the columns of $F_p$
	represent eigenvectors of the matrix $\crc ([v])$. Now, as a consequence of Remark \ref{rem1.5}, we can conclude the following proposition.
	
	\begin{prop}\label{prop1.4}
		Let $[v]$ be a symmetric/Hermitian tubular tensor of length $p$ and $D$ be the diagonal matrix
		$D=F^H_p\crc([v])F_p $
		where $F_p$ denotes the DFT matrix of order $p$. Then the tubular tensor
		$[v]$ is  positive definite iff the diagonal matrix $D$ has real positive diagonal entries.
	\end{prop}
	
	\begin{proof}
		According to relation
		\eqref{eq:bil} and Definition~\ref{def:spd},  the tubular tensor
		$[v]$ is SPD iff $\crc([v])$ is SPD. Since  $\crc([v])$ and $D$ are similar matrices, the tubular tensor
		$[v]$  is SPD  iff  $D$ has positive diagonal entries.  
	\end{proof}

	Let $[v]$ be any tubular tensor of length $p$. By Proposition \ref{prop1.4},
	there exists a diagonal matrix $D={\rm diag}(d_{11},d_{22},\ldots,d_{pp})$ such
	that $\crc([v])=F_pDF^H_p$. From this, one can define the function $f(C)$ as
	$f(C) = F_p f(D) F_p^H$ where $C=\crc([v])$. By properties of circulant
	matrices, it turns out that $f(C)$ is also circulant and therefore equal to
	$\crc (w)$ for some vector $w$. According to our notation \eqref{eq:circMat},
	that vector $w$ is just the first column of the related circulant matrix. Thus,
	we can define \eq{eq:fofv} f([v]) = [w] \quad \mbox{where} \quad w = F_p f(D)
	F_p^H e_1 .  \en Assume now that $[v]$ is Hermitian positive definite in which
	case $d_{ii}>0$ for $i=1,2,\ldots,p$.  For the particular case of the square
	root function, note that $\crc([v])=F_pD^{1/2}F_p^HF_pD^{1/2}F^H_p$ where
	$D^{1/2}:={\rm
		diag}(\sqrt{d_{11}},\sqrt{d_{22}},\ldots,\sqrt{d_{pp}})$. Therefore, we have
	$[v]=[w]\ast [w]$ where $\crc([w])=F_pD^{1/2}F_p^H$ as seen above.  In this case,
	$[w]$ is clearly a Hermitian positive definite tubular tensor and we will refer
	to it as the ``square root" or ``root'' of $[v]$ and denote it  by $[v]^{1/2}$.  
	Further exploring \eqref{eq:fofv} and considering  the matrix $F_p$ in
	\eqref{eq:FpMat}, reveals that the vector $F_p^H e_1 $ in \eqref{eq:fofv}
	equals $p^{-1/2} e $ where $e$ is the vector of all ones.  The result of
	$f(D) e/\sqrt{p}$ is just the vector with components $f(d_i) / \sqrt{p}$.
	Therefore, \emph{the vector $w$ is nothing but the Discrete Fourier Transform of
		the vector with components $f(d_i)$ scaled by $\sqrt{p}$}.

	We end this section by recalling the following theorem which presents the
	T-Jordan Canonical Form (TJCF) see \cite{Miao2021} for details.
	
	\begin{thm}{\rm (TJCF)}\label{thm:TJCF} 
		Let $\mathscr{A} \in \mathbb{C}^{n\times n \times p}$. There
		exist a non-singular tensor $\mathscr{P} \in \mathbb{C}^{n\times n \times p}$ and a F-upper-bi-diagonal\footnote{A tensor $\mathscr{A}\in \C^{n \times n \times p}$ is said to be  F-diagonal, or F-upper(-bi)-diagonal, if each frontal slice of $\mathscr{A}$ is (respectively) a diagonal and  upper(-bi)-diagonal matrix.} tensor $\mathscr{J} \in \mathbb{C}^{n\times n \times p}$ such
		that
		$\mathscr{A} = \mathscr{P}^{-1} \ast \mathscr{J} \ast \mathscr{P}.$
	\end{thm}

	\section{Eigenvalues of tensor with respect to T-product}\label{sec3}
	
	We begin this section by reviewing the notions of T-eigenvalue and eigentuple of
	a third-order tensor and then present the definition of tubular eigenvalues for
	tensors, along with a few properties. Furthermore, we will show how these three
	notions of eigenvalue are related. Finally, we briefly discuss a few additional
	properties of tubular eigenvalues for Hermitian tensors.
	
	\subsection{T-eigenvalues and eigentuples of a tensor}
	
	The following definition of T-eigenvalue was recently given in
	\cite{Liu}.
	\begin{defn}\label{def:Tevalue}
		Let $\mathscr{A} \in \mathbb{C}^{n\times n \times p}$.  If there exists a
		scalar $\lambda\in \mathbb{C}$ and a nonzero tensor
		$\mathscr{X} \in \mathbb{C}^{n\times 1 \times p}$ such that
		\[
		\mathscr{A} \ast \mathscr{X} = \lambda \mathscr{X}
		\]
		then the scalar $\lambda$ is called a T-eigenvalue of $\mathscr{A} $ and
		$\mathscr{X}$ is a T-eigenvector of $\mathscr{A} $ associated with $\lambda$.
	\end{defn}
	
	With this definition, Liu and Jin \cite{Liu} proved a number of inequalities for
	Hermitian tensors.  More precisely, Weyl's theorem and Cauchy's interlacing
	theorem were generalized for the tensor case with respect to T-eigenvalues.
	
	It is worth noting that the above definition is strongly related to the
	``$\bcrc$" representation of the tensor. In fact, if  $\lambda$
	is  a T-eigenvalue of $\mathscr{A} $, then 
	$$\bcrc(A)\ufld(\mathscr{X}) = \lambda \cdot \ufld(\mathscr{X}).$$
	That is, the T-eigenvalues of $\mathscr{A} $ are merely the eigenvalue of
	the matrix $\bcrc(A)$.
	The above definition of the T-eigenvalues is equivalent to a definition given
	in \cite{Miao, Miao2021} which relies on the
	T-Jordan Canonical Form seen earlier, see, Theorem~\ref{thm:TJCF}.
	Indeed, from the proof of Theorem~\ref{thm:TJCF}, it turns out that
	given a tensor $\mathscr{A} \in \mathbb{C}^{n\times n \times p}$  there exists
	a $p \times p$ {\rm (DFT)} matrix $F_p$ and upper triangular matrices
	$C_1,C_2,\ldots,C_p$ such that
	\[
	{\bcrc(\mathscr{J})} =(F_p \otimes I_n)  {\rm blockdiag(C_1,C_2,\ldots,C_p)}(F_p^H  \otimes I_n) . 
	\]
	Therefore, the articles \cite{Miao, Miao2021} define the
	diagonal elements of  the $C_i$'s as the T-eigenvalues of $\mathscr{A}$.

	It is more natural to define eigenvalues of a tensor as tubular tensors instead
	of scalars, see \cite{Braman,Kilmer2013} for details.  Specifically,
	Bradman \cite[Theorem 5.1]{Braman} defined the notion of real eigentuples and
	eigenmatrices for third-order tensors in $\mathbb{R}^{n\times n \times n}$. The
	subsequent definition of eigentuples was recently given by Qi and Zhang
	\cite[Section 3]{Qi} as an extension of the definition of eigentuples given in
	\cite{Braman}.
	
	Given an $n\times p$ matrix $X$, we consider the T-product $\cA \ast \cX $ where
	$\cX $ is the matrix $X$ viewed as an $n \times 1 \times p$ tensor. This
	T-product is itself an $n \times 1 \times p$ tensor which can be reshaped into
	an $n \times p$ matrix. We denote this matrix by $\cA \ast_m X $. This notation is used in the
	following definition \cite{Qi}.

	\begin{defn}\label{def:QiZh} 
		Suppose that $\mathscr{A} \in \mathbb{R}^{n\times n \times p}$. If there
		exists a nonzero matrix $X\in \mathbb{C}^{n\times p}$, $d \in \mathbb{C}^{ p}$
		such that
		\eq{eq:QiZh} 
		\mathscr{A} \ast_m X = X {\rm circ}([d])
		\en
		then the vector $d$  is called  eigentuple of $\mathscr{A}$, and the matrix $X$ is an eigenmatrix of $\mathscr{A}$ associated with the eigentuple $d$.\footnote{We comment that Qi and Zhang  \cite{Qi} used the notation $d \circ X$ to refer the product $ X {\rm circ}([d])$.}
	\end{defn}

	\subsection{Tubular eigenvalue} 
	Let $\mathscr{A} \in \mathbb{C}^{n\times n \times p}$ be given.  Here, we
	consider an eigenvalue of $\mathscr{A}$ as a tubular tensor and analyze the
	associated tubular spectrum.  The \textit{tubular eigenvalue} of $\mathscr{A}$
	is defined next.
	
	\begin{defn}\label{eig:def}The tubular tensor $[\lambda] \in \mathbb{C}^{1\times 1 \times p}$ is called a
		tubular eigenvalue of $\mathscr{A} \in \mathbb{C}^{n\times n \times p}$,  if there exists a tensor
		$\mathscr{X} \in \mathbb{C}^{n\times 1 \times p}$ such that
		\eq{eq:MainDef}
		\mathscr{A} \ast \mathscr{X} = \mathscr{X} \ast [\lambda]
		\en
		and the tubular tensor $\mathscr{X}^H \ast \mathscr{X}$ is non-singular. The
		tensor $\mathscr{X}$ is called the \emph{eigentensor} associated with the tubular
		eigenvalue $[\lambda]$. The \emph{tubular spectrum}, denoted by 
		$\sigma_T(\mathscr{A})$, is the set  of  all
		tubular eigenvalues of $\mathscr{A}$. 
	\end{defn}

	Note that in the above definition the tubular tensor
	$\mathscr{X}^H \ast \mathscr{X}$ associated with the eigentensor
	$\mathscr{X} \in \mathbb{C}^{n\times 1 \times p}$ must be non-singular. For $\mathscr{X} \in \mathbb{C}^{n\times 1 \times p}$, it is
	known that (see Eq. \eqref{bld})
	\[
	\bcrc(\mathscr{X})=(F_p\otimes I_n) {\rm blockdiag}(\tilde{x}_1,\tilde{x}_2,\ldots,\tilde{x}_p)F_p^H
	\]
	where $\tilde{x}_i$ are $n$ dimensional vectors for $i=1,2,\ldots,p$
	considered as an $n\times 1 $ matrix in the above expression involving  $\text{blockdiag}$.
	On can observe that 
	\[
	\bcrc(\mathscr{X})^H\bcrc(\mathscr{X})=
	F_p{\rm diag}(\tilde{x}_1^H\tilde{x}_1,\tilde{x}_2^H\tilde{x}_2,\ldots,\tilde{x}_p^H\tilde{x}_p)F_p^H .
	\]
	Notice that the tubular tensor $\mathscr{X}^H \ast \mathscr{X}$ is
	non-singular if and only if the $p \times p$ matrix
	$\bcrc(\mathscr{X})^H\bcrc(\mathscr{X})$ is non-singular. As a result, we see
	that the tubular tensor $\mathscr{X}^H \ast \mathscr{X}$ is non-singular iff
	the vectors $\tilde{x}_1,\tilde{x}_2,\ldots,\tilde{x}_p$ are all non-zero.
	
	%
	
	
	Let $x_1, x_p,\ldots, x_2$ be the vectors in the first bock row of the matrix
	$ \bcrc(\mathscr{X})$. The above considerations suggest that the condition of nonsingularity of
	$\mathscr{X}^H \ast \mathscr{X}$ may be somehow related to some form of linear independence of the
	columns of the matrix $X=[x_1,x_2,\ldots,x_p]$. We now explore this link.  As was just seen
	this condition is equivalent to the requirement that the matrix
	$\bcrc(\mathscr{X})^H\bcrc(\mathscr{X})$ be non-singular, which is turn is
	true iff the matrix
	$\bcrc(\mathscr{X})$, which is  of dimension $(np) \times p$, 
	is of full rank.  Assume now that there is a nonzero vector
	$a = [\alpha_1, \alpha_2, \alpha_{3}, \ldots, \alpha_p]^T$ such that
	$\bcrc(\mathscr{X}) \hat a = 0$, where
	$\hat a = [\alpha_1, \alpha_p, \alpha_{p-1}, \ldots, \alpha_2]^T$.
	The  indexing of the components of $\hat a$ is  intended
	to match that of the vectors $x_1, x_p,\ldots, x_2$.
	The first row of the relation $\bcrc(\mathscr{X}) \hat a = 0$ can be written as follows: 
	\[ \alpha_1 x_1 + \alpha_{p} x_{p} + \alpha_{p-1} x_{p-1} + \cdots + \alpha_{2}
	x_2 = 0.\] The second row sees the vectors $x_i$ circularly shifted
	\emph{down} (index increases by one cyclically). If we wish the linear
	combination to be in the same order of the vectors as above, we could keep the
	same indexing for the $x_i$'s but move \emph{up} the indices of the the
	$\alpha_i$ cyclically. This yields:
	\[
	\alpha_{1} x_2  + \alpha_p x_1 + \alpha_{p-1} x_p + \cdots +  \alpha_{p-2} x_{p-1}   = 0 \ \to \ 
	\alpha_p x_1 + \alpha_{p-1} x_p + \alpha_{p-2} x_{p-1} + \cdots + \alpha_{1} x_2  = 0 .\]
	This is repeated a few times until we reach the last row:
	\[ \alpha_2 x_1 + \alpha_{1} x_p + \alpha_{p} x_{p-1} + \cdots + {
		\alpha_{3}} x_2 = 0 .\] 
	These equations can be captured by the single matrix
	equation: $ [x_1, x_p,\ldots, {x_2}] h(\hat a) = 0 $ in which $h(\hat a) $
	is the Hankel matrix whose first column is $\hat a$ and columns $2, 3,\ldots, p$
	are obtained by moving repeatedly this column up in a circular fashion.  If we
	permute the columns of the matrix $[x_1, x_p,\ldots, x_1]$ into
	$X=[x_1, x_2,\ldots, x_p]$ and permute the rows of $h(\hat a)$ accordingly,
	then we arrive at the relation: \eq{eq:nonsingvec} X \crc(a) = 0.  \en In other
	words, the singularity of the matrix $\bcrc(\mathscr{X})^H\bcrc(\mathscr{X})$ is
	equivalent to the existence of a nonzero vector $a$ such that $X \crc(a) = 0$.  This
	can also be expressed in terms of the tubular tensor $[\hat a]$ since the
	condition $\bcrc(\cX) \hat a = 0$ translates to $\cX * [\hat a] = 0$. Thus,
	$\bcrc(\mathscr{X})^H\bcrc(\mathscr{X})$ is nonsingular iff there are no tubular
	tensors $\hat a $ for which $\cX * [\hat a] = 0$.
	
	
	It is known that the eigenvectors associated with distinct eigenvalues of a
	matrix, form a linearly independent set of vectors.  Proposition
	\ref{prop3.7n} to be stated later, will reveal that an analogous result can be
	extended to tubular eigenvalues. Before presenting the proposition, we need to
	recall the definition of T-linear combination \cite[Definition
	4.1]{Kilmer2013}, specify the concept of T-linear (in)dependency and establish
	a simple lemma.

	Given tubular tensors $[c_1],[c_2],\ldots,[c_k]$ of length $p$, the
	\textit{T-linear combination} of tensors
	$\mathscr{X}_1,\mathscr{X}_2,\ldots,\mathscr{X}_k$ of size $n\times 1 \times p$
	is defined by
	\[
	\mathscr{X}_1 \ast [c_1]+ \mathscr{X}_2 \ast [c_2]+ \cdots {+} \mathscr{X}_k \ast[c_k]\equiv\mathscr{X} \ast \mathscr{C}
	\]
	where $\mathscr{X}$ and $\mathscr{C}$ are respectively tensors of size
	$n \times k \times p$ and $k \times 1 \times p$ such that
	$\mathscr{X}_{:j:}=\mathscr{X}_j$ and $\mathscr{C}_{j::}=[c_j]$ for
	$j =1,2,\ldots,k$.  A T-linear combination of tensors of size
	$n\times 1 \times p$ can be viewed   as a generalization of  the  linear combination
	of vectors where tubular tensors play  the role of scalars.  In a natural way, we define the notion of linear
	(in)dependence for a set of tensors
	$\mathscr{X}_1,\mathscr{X}_2,\ldots,\mathscr{X}_k$ of size $n\times 1 \times p$
	as follows:
	\begin{defn}\label{def:lin}
		Let $\mathscr{X}_j$  be an $n\times 1 \times p$  tensor for $j =1,2,\ldots,k$.
		Assume that there exists tubular tensors $[c_1],[c_2],\ldots,[c_k]$ such that
		\begin{equation}\label{eqlin}
			\mathscr{X}_1 \ast [c_1]+ \mathscr{X}_2 \ast [c_2]+ \cdots +\mathscr{X}_k \ast[c_k] = \mathscr{O}.
		\end{equation}
		Tensors $\mathscr{X}_1,\mathscr{X}_2,\ldots,\mathscr{X}_k$ are said to be
		T-linearly dependent, if at least one of the tubular tensors $[c_1],[c_2],\ldots,[c_k]$ is non-zero.
		Otherwise, the tensors $\mathscr{X}_1,\mathscr{X}_2,\ldots,\mathscr{X}_k$ are called T-linearly independent.
	\end{defn}

	The following simple lemma helps clarify the definition.
	
	\begin{lem}
		Let  $\mathscr{X}_j$ be a tensor of size $n\times 1 \times p$ for $j=1,2,\ldots,k$ $(k\le n)$. 
		Assume that 
		\[
		\mathcal{S}_{\ell} :=\{\tilde{x}_{\ell} ^{(1)},\tilde{x}_{\ell}^{(2)},\ldots,\tilde{x}_{\ell}^{(k)}\}, \quad 1\le \ell \le p,
		\]
		where
		\begin{equation}\label{circ:form1}
			\bcrc(\mathscr{X}_j)=(F_p \otimes I_n) {\rm blockdiag}\left(\tilde{x}_1^{(j)},\tilde{x}_2^{(j)},\ldots,\tilde{x}_p^{(j)}\right)F_p^H\quad j=1,2,\ldots,k.
		\end{equation}
		The set of tensors $\mathscr{X}_1,\mathscr{X}_2,\ldots,\mathscr{X}_k$ is T-linearly independent if and only if each 
		$\mathcal{S}_{\ell} $ is a linearly independent set of vectors for  $\ell=1,2,\ldots,p$.
	\end{lem}
	
	\begin{proof}
		Note that \eqref{eqlin} holds iff
		\begin{equation}\label{circ:form2}
			\sum_{i=1}^{k}\bcrc(\mathscr{X}_i ) \crc([c_i])e_1 = 0
		\end{equation}
		where $0$ denotes the zero vector of size $np$ and  $e_1$ stands for the first column of the identity matrix of size $p$. It is known that there exist scalars $\tilde{c}_1^{(i)},\tilde{c}_2^{(i)},\ldots,\tilde{c}_p^{(i)}$ such that
		\[
		\crc([c_i])=F_p \text{diag}(\tilde{c}_1^{(i)},\tilde{c}_2^{(i)},\ldots,\tilde{c}_p^{(i)})F_p^H.
		\]
		In view of \eqref{circ:form1}, it can be verified that \eqref{circ:form2} is equivalent to
		\[
		(F_p\otimes I_n) {\rm blockdiag}\left(\sum_{i=1}^{k}\tilde{c}_1^{(i)}\tilde{x}_1^{(i)},\sum_{i=1}^{k}\tilde{c}_2^{(i)}\tilde{x}_2^{(i)},\ldots,\sum_{i=1}^{k}\tilde{c}_p^{(i)}\tilde{x}_p^{(i)}\right)F^H_pe_1=0.
		\]
		One can observe that the above relation holds iff each block of the following block vector is zero
		\[
		\left(\sum_{i=1}^{k}\tilde{c}_1^{(i)}\tilde{x}_1^{(i)},\sum_{i=1}^{k}\tilde{c}_2^{(i)}\tilde{x}_2^{(i)},\ldots,\sum_{i=1}^{k}\tilde{c}_p^{(i)}\tilde{x}_p^{(i)}\right)^T.
		\] 
		The proof of the assertion follows.
	\end{proof}
	
	The following proposition can now be stated. The proof of the proposition uses mathematical induction and  straightforward	algebraic manipulations and it is therefore omitted.

	\begin{prop}\label{prop3.7n}
		Let $\mathscr{A}$ be an $n\times n \times p$ tensor. Assume that $[\lambda_1],[\lambda_2],\ldots,[\lambda_k]$
		are $k$ tubular eigenvalues of $\mathscr{A}$ with
		associated eigentensors  $\mathscr{X}_1,\mathscr{X}_2,\ldots,\mathscr{X}_k$  for $k>1$.
		If $[\lambda_i]-[\lambda_j]$ is non-singular for $i,j=1,2,\ldots,k$ with $i\ne j$, then the eigentensors
		$\mathscr{X}_1,\mathscr{X}_2,\ldots,\mathscr{X}_k$ 	are T-linearly independent.
	\end{prop}

	Given a tubular tensor $[a]$, we define  the set 
	$\mathcal{S}_{[a]}=\{[a]_{::1},[a]_{::2},\ldots,[a]_{::p}\}$.  Let $[\lambda_1]$
	and $[\lambda_2]$ be two tubular eigenvalues of a tensor.
	We conclude this section with  a proposition which shows that if
	$\mathcal{S}_{[\lambda_1]}=\mathcal{S}_{[\lambda_2]}$ then
	$[\lambda_1]-[\lambda_2]$ is singular.
	
	\begin{prop}
		Let $[a]$ and $[b]$ be two tubular tensors of length $p$. If the sets
		$\mathcal{S}_{[a]}$ and $\mathcal{S}_{[b]}$ are equal, then $[a]-[b]$ is
		singular.
	\end{prop}
	\begin{proof}
		It is known that there exist diagonal matrices ${D^a}$
		and ${D^{b}}$ such that $\crc([a])=F_{p}{D^a}F_p^H$ and $\crc([b])=F_{p}{D^b}F_p^H$.
		Hence, $D^a_{11}=(1/p)e^T\crc([a])e$ and $D^b_{11}=(1/p)e^T\crc([b])e$  where $e$ is a vector of all ones. Therefore, the first diagonal elements of $D^a$ and ${D^b}$ are respectively given by
		\[
		D^a_{11}=\sum_{i=1}^{p}[a]_{::i}
		\quad
		{\rm and}
		\quad
		D^b_{11}=\sum_{i=1}^{p}[b]_{::i}.
		\]
		Since $\crc([a])-\crc([b])=F_{p}({D^a}-{D^b})F_p^H$, the proof follows
		from the fact that ${D^a}_{11}$ and ${D^b}_{11}$ are equal.
	\end{proof}
	
	\subsection{Relation between eigentuple and T-eigenvalue with tubular eigenvalue} 
	This subsection is devoted to establishing links between different definitions of
	T-eigenvalues and eigentuples with tubular eigenvalues of a tensor.
	The first result  shows a relation between eigentuples and  tubular eigenvalues
	of a tensor.
	
	\begin{prop}
		Let $\mathscr{A}$  be a complex tensor  of size $n\times n  \times p$ and  suppose
		that $d$  is an arbitrary  eigentuple of $\mathscr{A}$ with  associated 
		eigenmatrix $X$. Moreover,  assume that the tubular tensor  $[\lambda]$ and tensor
		$\mathscr{X}\in \mathbb{C}^{n\times 1 \times p}$ are defined such that
		\[
		[\lambda]_{111}=d_1 \quad {\rm and} \quad [\lambda]_{11i}=d_{p-i+2},~ {\rm for}~ i=2,3,\ldots,p
		\]
		and $ \ufld(\mathscr{X})=\text{vec}(X).  $ If
		$\mathscr{X}^H \ast \mathscr{X}$ is non-singular, then $[\lambda]$ is a
		tubular eigenvalue of $\mathscr{A}$ with corresponding eigentensor
		$\mathscr{X}$.
	\end{prop}
	
	\begin{proof}
		In view of Definition \ref{def:QiZh}, one can see that
		$\bcrc(\mathscr{A})\text{vec}(X)=(\crc([d])^T\otimes I_n)\text{vec}(X).$
		The assertion follows immediately  in view of  Definition \ref{eig:def} which
		can be rewritten as:
		\[
		\bcrc ( \cA ) \ufld(\cX) = (\crc ([\lambda]) \otimes I_{n}) \ufld(\cX).
		\]
	\end{proof}
	
	The proof of the  previous proposition also leads to the following  proposition.
	
	\begin{prop}
		Let $\mathscr{A}$ be a complex tensor of size $n\times n \times p$. Assume
		that $[\lambda]$ is an arbitrary tubular eigenvalue of $\mathscr{A}$ with the
		corresponding eigentensor $\mathscr{X}$. If the vector $d$ and matrix
		$X\in \mathbb{C}^{n\times p}$ satisfy
		\[
		d_1=	[\lambda]_{111} \quad {\rm and} \quad d_{i}=[\lambda]_{11{(p-i+2)}},~ {\rm for}~ i=2,3,\ldots,p
		\]
		and 
		$\text{vec}(X)=	\ufld(\mathscr{X}).$ Then, the vector $d$ is an eigentuple of $\mathscr{A}$ with the corresponding eigenmatrix $X$.
	\end{prop}

	The following theorem reveals that each tubular eigenvalue of
	$\mathscr{A} \in \mathbb{C}^{n\times n \times p}$ corresponds to $p$ of its
	T-eigenvalues.
	
	\begin{thm}\label{th3.4}
		Let $[\lambda]$ be an arbitrary tubular eigenvalue of $\mathscr{A}$. If $F_p$
		is the DFT matrix of order $p$ and
		\[
		F_p^H \crc([\lambda]) F_p ={\rm diag}(\tilde{\lambda}_1,\tilde{\lambda}_2,\ldots,\tilde{\lambda}_p)
		\]
		then $\tilde{\lambda}_1,\tilde{\lambda}_2,\ldots,\tilde{\lambda}_p$ are T-eigenvalues of $\mathscr{A}$.
	\end{thm}
	
	\begin{proof}
		Let $[\lambda]$ be a tubular eigenvalue of $\mathscr{A}$. Therefore, there exists a nonzero tensor $\mathscr{X}$ such that $\mathscr{A} \ast \mathscr{X} = \mathscr{X} \ast [\lambda]$.
		Considering the TJCF of $\mathscr{A}$, i.e., $\mathscr{A} = \mathscr{P}^{-1} \ast \mathscr{J} \ast
		\mathscr{P}$, we have 
		\[
		\mathscr{J} \ast \mathscr{Y}  =\mathscr{Y} \ast [\lambda]
		\]
		where $\mathscr{Y}=\mathscr{P} \ast \mathscr{X}$. The above relation is equivalent to
		\begin{eqnarray}
			\nonumber 	\bcrc(\mathscr{J}) \ufld(\mathscr{Y}) & = & 	\bcrc(\mathscr{Y}) \ufld([\lambda]) \\
			& = & \bcrc(\mathscr{Y}) {\rm circ}([\lambda]) e_1 \label{eq3}
		\end{eqnarray}
		as before, $e_1$ denotes the first column of the identity matrix of size
		$p$. It is known that
		\begin{equation}\label{eq4n}
			{\bcrc(\mathscr{J})} =(F_p \otimes I_n)  {\rm blockdiag(C_1,C_2,\ldots,C_p)}(F_p^H  \otimes I_n).
		\end{equation}
		The above relation together with \eqref{eq3} imply that
		\begin{equation}\label{eq4}
			{\rm blockdiag(C_1,C_2,\ldots,C_p)} (F_p^H \otimes I_n) \bcrc(\mathscr{Y}) F_pF_p^He_1=(F_p^H \otimes I_n) \bcrc(\mathscr{Y}) F_pF_p^H{\rm circ}([\lambda])F_pF_p^H e_1
		\end{equation}
		To simplify notation, we define the following $np \times p$ block diagonal matrix
		$$Z:=(F_p^H \otimes I_n) \bcrc(\mathscr{Y})F_p = {\rm blockdiag} (\tilde{z}_1,\tilde{z}_2,\ldots,\tilde{z}_p)$$
		where $\tilde{z}_1,\tilde{z}_2,\ldots,\tilde{z}_p$ are complex vectors of size $p$. Multiplying both sides of 
		\eqref{eq4} by $\sqrt{p}$, since $\sqrt{p} F^H_pe=(1,1,\ldots,1)^T$, we get:
		\[\left( {\begin{array}{*{20}{c}}
				{{C_1}}&{}&{}&{}\\
				{}&{{C_2}}&{}&{}\\
				{}&{}& \ddots &{}\\
				{}&{}&{}&{{C_p}}
		\end{array}} \right)\left( {\begin{array}{*{20}{c}}
				{{\tilde{z}_1}}\\
				{{\tilde{z}_2}}\\
				\vdots \\
				{{\tilde{z}_p}}
		\end{array}} \right) = \left( {\begin{array}{*{20}{c}} {{\tilde{z}_1}}&{}&{}&{}\\
				{}&{{\tilde{z}_2}}&{}&{}\\
				{}&{}& \ddots &{}\\
				{}&{}&{}&{{\tilde{z}_p}}
		\end{array}} \right)\left( {\begin{array}{*{20}{c}}
				{{\tilde{\lambda} _1}}\\
				{{\tilde{\lambda} _2}}\\
				\vdots \\
				{{\tilde{\lambda} _p}}
		\end{array}} \right)
		\]
		which is equivalent to saying that
		$C_i \tilde{z}_i = \tilde{\lambda}_i \tilde{z}_i$ for
		$i=1,2,\ldots,p$.  The assertion follows from the
		definition of T-eigenvalue.
	\end{proof}
	
	\begin{rem}\label{rem:enum}
		Consider the TJCF of $\mathscr{A}$, i.e,
		$\mathscr{A}=\mathscr{P}^{-1}\ast \mathscr{J} \ast \mathscr{P}$.  Assume that
		the matrices $C_1,C_2,\ldots,C_p$ are defined by \eqref{eq4n}.  Let
		$\tilde{\lambda}_i$ be an arbitrary eigenvalue of
		$C_i\in \mathbb{C}^{n\times n}$ with associated eigenvector $\tilde{z}_i$ for
		$i=1,2,\ldots,p$.  From the proof of the above theorem, it is not difficult to
		verify that one can also associate a tubular eigenvalue $[\lambda]$ to an
		arbitrary given set of eigenvalues
		$\tilde{\lambda}_1,\tilde{\lambda}_2,\ldots,\tilde{\lambda}_p$.  Notice that
		each $C_i$ has at most $n$ eigenvalues. Therefore, the total possible number of
		tubular eigenvalues of an $n\times n \times p$ tensor is $n^p$.
	\end{rem}
	
	The following proposition which is an immediate consequence of Theorem \ref{th3.4}
	extends a well-known result on the spectra of the products of two
	matrices. It should be noted that the proposition can be directly proved
	without considering the link between tubular eigenvalues and T-eigenvalues.
	
	\begin{prop}\label{commute:eig2}
		Let $\mathscr{A}$ and $\mathscr{B}$ be tensors of size $n\times n \times p$. Then,
		\[
		\bar{\sigma}_T(\mathscr{A}\ast \mathscr{B})={\bar \sigma_T}(\mathscr{B}\ast \mathscr{A})
		\]
		where $\bar{\sigma}_T(\mathscr{W})$ stands for the set of all non-singular
		tubular eigenvalues of $\mathscr{W}$.
	\end{prop}

	
	\subsection{Tubular eigenvalues of Hermitian tensors}
	In this section, we study the
	properties of tubular eigenvalues of Hermitian tensors.
	
	Let $[\lambda]$ be an arbitrary tubular eigenvalue of
	a Hermitian tensor $\mathscr{A} \in \mathbb{C}^{n\times n \times p}$.
	There exists an eigentensor $\mathscr{X}$ such that
	$ \mathscr{A} \ast \mathscr{X} = \mathscr{X} \ast [\lambda].  $
	Therefore, recalling the definition of the square root  of tubular tensors
	seen at the end of Subsection~\ref{sec:basics}, we can write: 
	\[
	\mathscr{Z}^H \ast \mathscr{A} \ast \mathscr{Z}=   [\lambda]
	\]
	where $\mathscr{Z}=\mathscr{X} \ast (\mathscr{X}^H
	\ast\mathscr{X})^{-1/2}$. A direct consequence of the above relation is that 
	tubular eigenvalues of a Hermitian tensor are Hermitian tubular
	tensors. Now consider the following decomposition of $\mathscr{A}$,
	\begin{equation}\label{bloc:diag}
		\bcrc(\mathscr{A})=(F_p\otimes I_n){\rm blockdiag}(\tilde{A}_1,\tilde{A}_2,\ldots,\tilde{A}_p)(F_p^H\otimes I_n).
	\end{equation}
	Notice that the block matrices $\tilde{A}_1,\tilde{A}_2,\ldots,\tilde{A}_p$ are Hermitian. This ensures the existence of  diagonal matrices $\tilde{D}_1,\tilde{D}_2,\ldots,\tilde{D}_p$ and unitary matrices  $\tilde{Q}_1,\tilde{Q}_2,\ldots,\tilde{Q}_p$
	such that
	\begin{equation}\label{diag1}
		\tilde{A}_i=\tilde{Q}_i^H \tilde{D}_i\tilde{Q}_i, \quad \text{for} \quad i=1,2,\ldots,p.
	\end{equation}
	Straightforward computations reveal that
	\begin{eqnarray*}
		\bcrc(\mathscr{A}) &= &	(F_p\otimes I_n){\rm blockdiag}(\tilde{Q}_1^H,\tilde{Q}_2^H,\ldots,\tilde{Q}_p^H)(F_p^H\otimes I_n)(F_p\otimes I_n){\rm blockdiag}(\tilde{D}_1,\tilde{D}_2,\ldots,\tilde{D}_p)\\
		&&	(F_p^H\otimes I_n)(F_p\otimes I_n){\rm blockdiag}(\tilde{Q}_1,\tilde{Q}_2,\ldots,\tilde{Q}_p)(F_p^H\otimes I_n)\\
		&= &	\bcrc(\mathscr{Q}^H)\bcrc(\mathscr{D})\bcrc(\mathscr{Q}).
	\end{eqnarray*}
	Equivalently, we have
	\begin{equation}\label{udiag}
		\mathscr{A} =\mathscr{Q}^H \ast \mathscr{D} \ast \mathscr{Q},
	\end{equation}
	where $\mathscr{Q}$ and $\mathscr{D}$ are respectively  unitary and F-diagonal tensors. From Eq. \eqref{udiag}, it is follows immediately that
	\[
	\mathscr{A} \ast \mathscr{Z}_i=\mathscr{Z}_i\ast [\mathscr{D}(i,i,:)]
	\]
	where $\mathscr{Z}_i=\mathscr{Q}^H(:,i,:)$ for $i=1,2,\ldots,n$. In fact,
	tensors $ [\mathscr{D}(1,1,:)],[\mathscr{D}(2,2,:)],\ldots,[\mathscr{D}(n,n,:)]$
	are tubular eigenvalues of $\mathscr{A}$.  Without loss of generality, we may
	assume that eigenvalues of each $\tilde{D}_i$ ($1\le i \le p$) are labeled in
	increasing order. That is, we consider the case where
	$\tilde{D}_i={\rm
		diag}(\tilde{d}_1^{(i)},\tilde{d}_2^{(i)},\ldots,\tilde{d}_n^{(i)})$ with
	\begin{equation}\label{ordered:diag}
		\tilde{d}_1^{(i)}\le \tilde{d}_2^{(i)}\le \cdots \le \tilde{d}_n^{(i)}, \quad \text{for} \quad i=1,2,\ldots,p.
	\end{equation}
	Notice that 
	\[
	[\mathscr{D}(j,j,:)] =\mathscr{E}_j^T\ast \mathscr{D} \ast \mathscr{E}_j
	\]
	where $\mathscr{E}_j=\mathscr{I}_{nnp}(:,j,:)$ for $j=1,2,\ldots,n$. It is not difficult to verify that
	\[
	\crc([\mathscr{D}(i,i,:)])=F_p {\rm diag}(\tilde{d}_i^{(1)},\tilde{d}_i^{(2)},\ldots,\tilde{d}_i^{(p)})F_p^H\quad \text{for} \quad i=1,2,\ldots,n.
	\] 
	Therefore, we can observe that
	\begin{equation}\label{order}
		[\mathscr{D}(i,i,:)] \preceq 	[\mathscr{D}(j,j,:)]\quad \text{for} \quad i\le j.
	\end{equation}
	To simplify notation, we denote $[\mathscr{D}(i,i,:)]$ by $[\lambda_i]$ for
	$i=1,2,\ldots,n$. From the relation \eqref{order}, it turns out that we have $n$
	tubular eigenvalues $[\lambda_1], [\lambda_2], \ldots, [\lambda_n]$ of the
	Hermitian tensor $\mathscr{A}$ with corresponding eigentensors
	$\mathscr{Z}_1,\mathscr{Z}_2,\ldots,\mathscr{Z}_n$ such that
	\begin{equation}\label{ordered:eig}
		([\lambda_m(\mathscr{A})]:=)[\lambda_1] \preceq [\lambda_2] \preceq \cdots \preceq [\lambda_n]:(=[\lambda_M(\mathscr{A})]).
	\end{equation}
	
	In Remark \ref{rem:enum}, it was observed that the total number of tubular eigenvalues of $\mathscr{A}$ is more than $n$. However, one can verify that the following relation holds for any tubular eigenvalue $[\lambda]$ of the Hermitian tensor $\mathscr{A}$,
	\[
	[\lambda_m(\mathscr{A})] \preceq [\lambda] \preceq [\lambda_M(\mathscr{A})].
	\]
	Let $\lambda_{\min}(\tilde{A}_i)$ and $\lambda_{\max}(\tilde{A}_i)$
	be the extreme eigenvalues of $\tilde{A}_i$  for $i=1,2,\ldots,p$ where $\tilde{A}_1,\tilde{A}_2,\ldots,\tilde{A}_p$ satisfy in \eqref{bloc:diag}.
	It is not difficult to see that by Theorem \ref{th3.4} 
	when $\mathscr{A}$ is Hermitian, then
	\[
	F_P^H\crc[\lambda_m(\mathscr{A})]F_p=
	\text{diag}(\lambda_{\min}(\tilde{A}_1),\lambda_{\min}(\tilde{A}_2),\ldots,\lambda_{\min}(\tilde{A}_p))
	\]
	and
	\[
	F_P^H\crc[\lambda_M(\mathscr{A})]F_p= \text{diag}(\lambda_{\max}(\tilde{A}_1),\lambda_{\max}(\tilde{A}_2),\ldots,\lambda_{\max}(\tilde{A}_p)).
	\]
	Suppose that $A$ and $B$ are two Hermitian matrices. The following two inequalities are
	consequences of Weyl's Theorem in the matrix case,
	\begin{eqnarray*}
		\lambda_{\max}(A+B) & \leq & \lambda_{\max}(A) + \lambda_{\max}(B),\\
		\lambda_{\min}(A+B) & \geq & \lambda_{\min}(A) + \lambda_{\min}(B).
	\end{eqnarray*}
	
	In view of  the earlier discussion of this subsection, we infer
	the next proposition which extends the above relations to tubular tensors.

	\begin{prop}\label{prop:eig1}
		Let $\mathscr{A}$ and $\mathscr{B}$ be two Hermitian tensors. If $[\lambda]$ is an arbitrary 
		tubular eigenvalue of $\mathscr{A}+\mathscr{B}$, then
		$
		\lambda_m(\mathscr{A}) +\lambda_m(\mathscr{B})	\preceq [\lambda] \preceq \lambda_M(\mathscr{A}) +\lambda_M(\mathscr{B}) .
		$
	\end{prop}
	
	The following lemma provides lower and upper bounds for extreme eigenvalues of the product of two matrices under certain conditions, see \cite{Zhang} for the proof.
	
	\begin{lem}\label{lem2} 
		Suppose that $A$ is a Hermitian negative definite matrix and
		$B$ is Hermitian positive semidefinite. Then the eigenvalues of $AB$
		are real and satisfy
		\[{\lambda _{\min}}(A){\lambda _{\min}}(B) \le {\lambda _{\max}}(AB)
		\le {\lambda _{\max}}(A){\lambda _{\min}}(B),\]
		\[{\lambda _{\min}}(A){\lambda _{\max}}(B) \le {\lambda _{\min}}(AB)
		\le {\lambda _{\max}}(A){\lambda _{\max}}(B).\]
	\end{lem}
	
	Lemma \ref{lem2} can be easily adapted to tubular eigenvalues as is shown next.
	
	\begin{prop}\label{prop:eig2}
		Let $\mathscr{A}$ and $\mathscr{B}$ be Hermitian negative definite and Hermitian positive (semi-)definite tensors, respectively. If $[\lambda]\in \sigma_T(\mathscr{A} \ast \mathscr{B})$, then $[\lambda]$ is Hermitian negative (semi-)definite tensors and 
		\[
		[\lambda_m(\mathscr{A} \ast \mathscr{B})] \preceq  [\lambda] \preceq [\lambda_M(\mathscr{A} \ast \mathscr{B})]
		\]
		where
		\[[\lambda_m(\mathscr{A})]\ast [\lambda_m(\mathscr{B}) ]\preceq [\lambda_M(\mathscr{A} \ast \mathscr{B})]
		\preceq [\lambda_M(\mathscr{A})] \ast [\lambda_m(\mathscr{B})],\]
		\[[\lambda_m(\mathscr{A})]\ast [\lambda_M(\mathscr{B})] \preceq [\lambda_m(\mathscr{A} \ast \mathscr{B})]
		\preceq [\lambda_M(\mathscr{A})]\ast [\lambda_M(\mathscr{B})].\]
	\end{prop}

	\noindent Let $\mathscr{P}$ be a preconditioner for the tensor equation
	$\mathscr{A} \ast \mathscr{X}=\mathscr{B}$. Propositions \ref{prop:eig1} and
	\ref{prop:eig2} can be used to study the tubular spectrum of the preconditioned
	tensor $\mathscr{P} \ast \mathscr{A}$. In particular, the provided relations in
	Proposition \ref{prop:eig1} can be helpful when $\mathscr{P}$ is extracted from
	$\mathscr{A}$.
	
	Using the well-known Courant--Fischer Min-Max principle \cite[Theorem 1.21]{Saad2003}
	for Hermitian matrices, we can prove the following proposition.
	
	\begin{prop}
		Let $\mathscr{A}\in \mathbb{C}^{n\times n \times p}$ be a Hermitian. Then, the following relation holds:
		\begin{equation}
			[\lambda_m(\mathscr{A})] \preceq (\mathscr{X}^H \ast\mathscr{A}\ast \mathscr{X})\ast \left(\mathscr{X}^H \ast \mathscr{X}\right)^{-1} \preceq [\lambda_M(\mathscr{A})] 
		\end{equation}
		for any tensor $\mathscr{X}\in \mathbb{C}^{n\times 1 \times p}$ provided that $\mathscr{X}^H \ast \mathscr{X}$ is non-singular
	\end{prop}
	\begin{proof}
		Consider the decomposition \eqref{bloc:diag} for $\bcrc(\mathscr{A})$ and let
		\[
		\bcrc(\mathscr{X})= (F_p \otimes I_n) {\rm blockdiag}(\tilde{x}_1,\tilde{x}_2,\ldots,\tilde{x}_p)F_p^H. 
		\]
		Evidently, we have
		\[
		\crc\left(\mathscr{X}^H \ast\mathscr{A}\ast \mathscr{X}\right)\crc\left((\mathscr{X}^H \ast \mathscr{X})^{-1}\right)=F_p\text{diag}\left(\frac{\tilde{x}^H_1\tilde{A}_1\tilde{x}_1}{\tilde{x}^H_1\tilde{x}_1},\frac{\tilde{x}^H_2\tilde{A}_2\tilde{x}_2}{\tilde{x}^H_2\tilde{x}_2},\ldots,\frac{\tilde{x}^H_p\tilde{A}_p\tilde{x}_p}{\tilde{x}^H_p\tilde{x}_p}\right)F_p^H.
		\]
		In view of the above relation, the assertion follows immediately by applying Courant--Fischer Min-Max principle for
		Hermitian matrices $\tilde{A}_1,\tilde{A}_2,\ldots,\tilde{A}_p$.
	\end{proof}
	
	We end this part by presenting the following two results on positive (semi)definite tensors which can be deduced from Theorem \ref{th3.4}. 
	
	\begin{prop}
		Let $\mathscr{B}\in \mathbb{C}^{n\times n \times p}$ be a Hermitian positive (semi-)definite tensor.
		If $\mathscr{A}\in \mathbb{C}^{n\times n \times p}$ is Hermitian positive definite, then all tubular eigenvalues of $\mathscr{A}\ast \mathscr{B}$ are Hermitian positive (semi-)definite. 
	\end{prop}
	
	\begin{thm}
		Let $\mathscr{A}\in \mathbb{C}^{n\times n \times p}$ be a Hermitian tensor. The tensor $\mathscr{A}$ is positive definite if and only if all tubular eigenvalues of  $\mathscr{A}$  are positive definite.
	\end{thm}
	
	\section{Tubular spectral analysis of tensor iterative methods}\label{sec4}
	
	Consider the following tensor equation \eq{eqten} \mathscr{A} \ast
	\mathscr{X}=\mathscr{B} \en where
	$\mathscr{X}\in \mathbb{R}^{n\times 1 \times p}$ is unknown, the tensors
	$\mathscr{A}\in \mathbb{R}^{n\times n \times p}$ and
	$\mathscr{B}\in \mathbb{R}^{n\times 1 \times p}$ are given. Iterative methods
	can be applied for solving \eqref{eqten} in \textit{tubular} and \textit{global}
	forms. Both of these versions for Krylov subspace methods have been exploited in
	the literature, see \cite{El2021tensor,ElIchi,Kilmer2013,Kilmer2011,Reichel}.
	Considering the decomposition \eqref{bloc:diag} for $\mathscr{A}$, the tubular
	version of an iterative method is mathematically equivalent to implementing it
	on $p$ subproblems (with coefficient matrices of size $n$) in the Fourier
	domain, e.g., see \cite[Subsection 6.2]{Kilmer2013} for more details. The global
	version of an iterative method refers to the case when the method is basically
	used for solving the linear system of equations
	$\bcrc(\mathscr{A}) \ufld(\mathscr{X})=\ufld(\mathscr{B})$ and in a practical
	implementation it is used in tensor structure, see, for instance, 
	\cite{El2021tensor}.
	
	It is known that when the matrix $A$ is symmetric then the distribution of
	eigenvalues of $A$ are descriptive for convergence analysis of some Krylov
	subspace methods to solve $Ax=b$, see \cite{Saad2003}. As a result, in the case
	when the tensor $\mathscr{A}$ is symmetric, the T-eigenvalues play a key role in
	the convergence analysis of global Krylov subspace methods to solve
	\eqref{eqten}.  To the best of our knowledge, the convergence properties of
	tubular Krylov subspace methods have not been discussed in detail. Besides, the
	tubular and global iterative methods have not been theoretically compared. We
	show that in exact arithmetic, a tubular version of an iterative method provides
	a more accurate approximation compared to its global form while both forms seek
	their new approximation in the same subspace. In addition, we present some
	results on tubular eigenvalues which can be used for convergence analysis of
	tubular iterative methods.  For better clarity, we present the discussions
	for stationary and a class of non-stationary methods in two separate parts. To
	be specific, in each part, we consider tubular and global forms of a simple
	iterative method and study the convergence properties of the tubular form.
	
	\subsection{Stationary iterative methods}
	Assume that the coefficient tensor $\mathscr{A}$ in \eqref{eqten} is a
	non-singular. Let the non-singular tensor $\mathscr{M}$ and the tensor
	$\mathscr{N}$ be given such that $\mathscr{A}=\mathscr{M}-\mathscr{N}$. A
	generic stationary iterative method produces the sequence of approximations
	$\{\mathscr{X}_k\}_{k=1}^{\infty}$as follows:
	\begin{equation}\label{method:sta}
		\mathscr{X}_{k+1}=\mathscr{G}\ast \mathscr{X}_{k}+\mathscr{M}^{-1}\ast \mathscr{B}, \quad k=0,1,2,\ldots
	\end{equation}
	where the initial guess $\mathscr{X}_{0}$ is given and
	$\mathscr{G}=\mathscr{M}^{-1} \ast \mathscr{N}$ is called the iteration tensor.
	
	In order to analyze the convergence of stationary iterative methods with
	respect to the tubular spectrum, we need to define the notion of tubular spectral
	radius. Note that the notion of spectral radius can be also extended to the tensor
	framework considering the definition of T-eigenvalue. Thus, the T-spectral
	radius is a positive scalar $\bar{\rho}_T(\mathscr{A})$ such that
	$|\lambda| \le \bar{\rho}_T(\mathscr{A})$ for any arbitrary T-eigenvalue
	$\lambda$ of $\mathscr{A}$. In fact, $\bar{\rho}_T(\mathscr{A})$ is the spectral
	radius of $\bcrc(\mathscr{A})$.
	
	\begin{defn}
		The \emph{tubular spectral radius} of
		$\mathscr{A}$ is defined as follows:
		$$[\rho_T(\mathscr{A})] := \{([\mu]^H \ast [\mu])^{\frac{1}{2}} ~|~ [\mu]^H \ast [\mu] \succeq  [\lambda]^H \ast [\lambda], \quad \forall [\lambda] \in \sigma_T(\mathscr{A}) \}.$$
	\end{defn}

	Given a (semi-)symmetric/Hermitian positive definite tubular tensor $[v]$, we
	observe that $v_{111} > (\ge)~ 0$. Consequently, we can deduce that
	$[a] \succ (\succeq) ~[b]$ implies that $a_{111} > (\ge) ~b_{111}$. This simple
	fact is helpful in proving the following lemma.
	\begin{lem}\label{lem2.4}
		Let $\{[ a_\ell ]\}_{\ell=0}^{\infty}$ be a sequence of complex tubular
		tensors of length $p$.  If there exists a Hermitian positive define tubular
		tensor $[w]$ such that $[w] \prec [e_1]$ and
		\begin{equation} \label{dec}
			[a_\ell]^H \ast [a_\ell] \preceq [w]\ast	[a_{\ell-1}]^H \ast [a_{\ell-1}]
		\end{equation}
		then $[a_\ell] \to 0$ as $\ell \to \infty$.
	\end{lem}
	
	\begin{proof}
		The assumption $[w] \prec [e_1]$ together with \eqref{dec}  ensure  the existence of non-negative constant $\omega <1$ such that
		\begin{equation} \label{dec2}
			[a_\ell]^H \ast [a_\ell] \preceq \omega ~	[a_{\ell-1}]^H \ast [a_{\ell-1}].
		\end{equation}
		With each tubular tensor $[a_\ell]$, we associate the scalar $s_\ell$ defined by 
		\begin{equation}\label{s1}
			{s_\ell } = \sum\limits_{j = 1}^p {|a_{\ell ,11j}|^2} 
		\end{equation}
		where $a_{\ell ,11j}$ denotes the $j$th frontal slice of
		$[a_\ell]$. Basically, $s_\ell$ is the entry in position $(1,1,1)$ of
		$[a_\ell]^H \ast [a_\ell]$. Therefore, the relation \eqref{dec2} implies that
		$\{ s_\ell \}_{\ell=1}^{\infty}$ is a monotonically decreasing sequence of
		non-negative scalars which shows that
		${\mathop {\lim }\limits_{\ell \to \infty }}{s_\ell }=\tau$ such that
		$\tau\ge 0$. In particular, the relation \eqref{dec2} guarantees the validity
		of the following relation
		\[
		0 \le s_\ell \le \omega  s_{\ell-1}.
		\]
		As a result, we deduce that $\tau=0$, i.e.,
		${\mathop {\lim }\limits_{\ell \to \infty }}{s_\ell }=0$. Now, by \eqref{s1},
		we can conclude that each of the frontal slices of the tubular tensor $[a_\ell]$
		goes to zero as $\ell \to \infty$ which completes the proof.
	\end{proof}
	
	Let
	${\mathscr{A}^k} := \underbrace {\mathscr{A}\ast \mathscr{A} \ast \ldots\ast
		\mathscr{A}}_{{\rm k-times}}$. In the following, we show that if
	$[\rho_T(\mathscr{A})]\prec [e_1]$ then ${\mathscr{A}^k}$ converges to zero as
	$k \to \infty$.
	
	\begin{prop}\label{prop3.6}
		Let $\mathscr{A}$ be a tensor of size $n \times n \times p$. If $[\rho_T(\mathscr{A})]\prec [e_1]$ then	${\mathop {\lim }\limits_{k \to \infty }}{\mathscr{A}^k } = \mathscr{O}.$
	\end{prop}
	
	\begin{proof}
		Let	$[\lambda]$ be an arbitrary tubular  eigenvalue of $\mathscr{A}$. We define  $[\lambda^k] := \underbrace {[\lambda]\ast [\lambda]\ast \ldots\ast [\lambda]}_{{\rm k-times}}$ with ${[\lambda^0]}=[e_1]$. Evidently, for $k\ge 1$, we have 
		\[
		[\lambda^k]^H\ast [\lambda^k]=([\lambda]^H\ast [\lambda])\ast ([\lambda^{k-1}]^H\ast [\lambda^{k-1}])
		\]	
		which implies that
		\begin{equation*}
			[\lambda^k]^H \ast [\lambda^k] \preceq ([\rho_T(\mathscr{A})]\ast [\rho_T(\mathscr{A})]) \ast [\lambda^{k-1}]^H \ast [\lambda^{k-1}]\quad k=1,2,\ldots,
		\end{equation*}
		
		Now, under the assumption $[\rho_T(\mathscr{A})]\prec [e_1]$, Lemma
		\ref{lem2.4} shows that $[\lambda^k]\to 0$ as $k\to \infty$. It can be
		verified that $[\lambda^k]$ is a tubular eigenvalue of
		${\mathscr{A}^k}=\mathscr{P}^{-1}\ast \mathscr{J}^{k} \ast \mathscr{P}$
		where ${\mathscr{A}}=\mathscr{P}^{-1}\ast \mathscr{J} \ast \mathscr{P}$
		is the TJCF of $\mathscr{A}$. It is not difficult to verify that
		$\mathscr{J}^k$ converges to zero as $k\to \infty$ which completes the
		proof.
	\end{proof}
	
	We can now present the following proposition which can be proved by using
	straightforward computations.
	
	\begin{prop}\label{prop3.7}
		Let $\mathscr{A}$ be a tensor of size $n \times n \times p$.  If $[\rho_T(\mathscr{A})]\prec [e_1]$, then  $\mathscr{I}-\mathscr{A}$ is non-singular and 
		\begin{equation*}
			(\mathscr{I}_{nnp}-\mathscr{A})^{-1}=\sum\limits_{k = 0}^\infty  {{\mathscr{A}^k}}.
		\end{equation*}
	\end{prop}  
	
	Using Propositions \ref{prop3.6} and \ref{prop3.7}, we  observe that the
	sequence of approximate solutions $\{\mathscr{X}_{k}\}_{k=1}^{\infty}$ produced
	by iteration \eqref{method:sta}  converges for any initial guess
	$\mathscr{X}_{0}$ provided that $[{\rho}_T(\mathscr{G})]\prec[e_1]$.
	
	Now, we consider the Richardson method for solving the following normal equation
	\eq{normal}
	\mathscr{A}^T \ast \mathscr{A} \ast \mathscr{X}=\mathscr{A}^T \ast \mathscr{B}.
	\en
	
	The tubular and global versions of the Richardson method are given as follows:
	\begin{equation}\label{richard1}
		\mathscr{X}_{k+1}=\mathscr{X}_{k}+\mathscr{A}^T \ast (\mathscr{B}-\mathscr{A} \ast \mathscr{X}_{k}) \ast [\alpha], \quad k=0,1,2,\ldots
	\end{equation}
	and 
	\begin{equation}\label{richard2}
		\mathscr{X}_{k+1}=\mathscr{X}_{k}+\mu ~\mathscr{A}^T \ast (\mathscr{B}-\mathscr{A} \ast \mathscr{X}_{k}), \quad k=0,1,2,\ldots
	\end{equation}
	where $[\alpha]$ and $\mu$  are respectively prescribed symmetric tubular tensor and positive scalar. 
	
	Given a tubular tensor $[a]$ of length $p$, we define tensor $\mathscr{D}_{[a]}$
	which refers to the tensor whose nonzero entries are given by
	$(\mathscr{D}_{[a]})_{ii:}=[a]$ for $i=1,2,\ldots,p$. Given an arbitrary tensor
	$\mathscr{X}^{n\times 1 \times p}$, it can be verified that
	\[\mathscr{D}_{[s]} \ast \mathscr{X} = \left( {\begin{array}{*{20}{c}}
			{[s]\ast {\mathscr{X}_{1::}}}\\
			{[s]\ast{\mathscr{X}_{2::}}}\\
			\vdots \\
			{[s]\ast{\mathscr{X}_{n::}}}
	\end{array}} \right) = \left( {\begin{array}{*{20}{c}}
			{{\mathscr{X}_{1::}}\ast [s]}\\
			{{\mathscr{X}_{2::}}\ast [s]}\\
			\vdots \\
			{{\mathscr{X}_{n::}}\ast [s]}
	\end{array}} \right) = \left( {\begin{array}{*{20}{c}}
			{{\mathscr{X}_{1::}}}\\
			{{\mathscr{X}_{2::}}}\\
			\vdots \\
			{{\mathscr{X}_{n::}}}
	\end{array}} \right)\ast [s]=\mathscr{X}\ast [s].\]
	This shows that the iteration \eqref{richard1} is equivalent to the following form:
	\begin{eqnarray*}
		\mathscr{X}_{k+1}&=& \mathscr{G}_{[\alpha]} \ast \mathscr{X}_{k}+\mathscr{D}_{[\alpha]} \ast \mathscr{A}^T \ast\mathscr{B}, \quad k=0,1,2,\ldots
	\end{eqnarray*}
	where $\mathscr{G}_{[\alpha]}=\mathscr{I}-\mathscr{D}_{[\alpha]} \ast \mathscr{A}^T\ast \mathscr{A}$.
	The above form can be used to establish sufficient condition on $[\alpha]$ that guarantee the convergence of \eqref{richard1}. Notice that \eqref{richard2} can be also rewritten in the following form:
	\begin{eqnarray*}
		\mathscr{X}_{k+1}&=&\mathscr{G}_{[\alpha_{\mu}]} \ast \mathscr{X}_{k}+\mathscr{D}_{[\alpha_{\mu}]} \ast \mathscr{A}^T \ast\mathscr{B}, \quad k=0,1,2,\ldots.
	\end{eqnarray*}
	where $[\alpha_{\mu}] = \mu[e_1]$ and
	$\mathscr{G}_{[\alpha_{\mu}]}=\mathscr{I}-\mathscr{D}_{[\alpha_{\mu}]} \ast
	\mathscr{A}^T\ast \mathscr{A}$. It is not difficult to verify that the tubular
	eigenvalues of $\mathscr{G}_{[\alpha]}$ and $\mathscr{G}_{[\alpha_{\mu}]}$ are
	respectively given by $[e_1]-[\lambda] \ast [\alpha]$ and $[e_1]-\mu [\lambda]$
	where $[\lambda] \in \sigma_T(\mathscr{A}^T\ast \mathscr{A})$. As a result, one
	can observe that the iterative methods \eqref{richard1} and \eqref{richard2}
	converge for 
	$[\alpha] \prec 2 [\lambda_M(\mathscr{A}^T \ast\mathscr{A})^{-1}]$ and
	$\mu < 2\bar{\lambda}_M^{-1}$ where $\bar{\lambda}_M$ denotes the maximum
	T-eigenvalue of $\mathscr{A}^T \ast \mathscr{A}$.

	The values of $[\alpha]$ and $\mu$ which yields the best {\em asymptotic}
	convergence rate\footnote{The term ``best" asymptotic convergence refers to the
		fact that
		$\rho_T(\mathscr{G}_{[\alpha^*]}) \preceq \rho_T(\mathscr{G}_{[\alpha]})$ for
		any symmetric positive definite tubular tensors and
		$\bar{\rho}(\mathscr{G}_{[\alpha_{\mu^*}]}) \le
		\bar{\rho}(\mathscr{G}_{[\alpha_{\mu}]})$ for any positive scalar $\mu$.} in
	terms of tubular and scalar spectral radii are respectively given by
	\begin{equation}\label{tubopt}
		[\alpha^*]=2([\lambda_M(\mathscr{A}^T \ast \mathscr{A})]+[\lambda_m(\mathscr{A}^T \ast \mathscr{A})])^{-1}
	\end{equation}
	and 
	\begin{equation}\label{richardopt}
		\mu^*=2{(\bar{\lambda}_M+\bar{\lambda}_m)}^{-1}
	\end{equation}
	where $\bar{\lambda}_M$ and $\bar{\lambda}_m$ denote the extreme T-eigenvalues of  $\mathscr{A}^T \ast \mathscr{A}$. 
	
	It can be shown that
	$\rho_T(\mathscr{G}_{[\alpha^*]}) \preceq \rho_T(\mathscr{G}_{[\mu^*]})$.
	Therefore, the sequence of approximations produced by \eqref{richard1} with
	$[\alpha^*]$ is expected to converge asymptotically  faster than the one
	generated by \eqref{richard2} with $\mu^*$. To numerically verify the
	superiority of \eqref{richard1} over \eqref{richard2}, some experimental results
	are reported in Subsection \ref{num}.

	\subsection{Orthogonal projection methods}  
	Here, we consider tubular and global forms of a non-stationary type of iterative
	methods. More precisely, the orthogonal projection technique is exploited to
	solve \eqref{normal} in which $\mathscr{A}$ is assumed to be non-singular. It is
	theoretically shown that the tubular form of the method outperform its global
	version. We do not discuss the minimum residual (MR)-based iterative methods.
	However,  similar comparison results can be established between tubular and
	global versions of MR-type methods for solving \eqref{eqten}.
	
	The main discussion of this part begins with presenting some basic concepts on
	the bilinear form as an extension for the notion of inner product. Then, to
	show the role of tubular eigenvalues in convergence analysis of tubular
	iterative method, we prove the convergence of tubular form of steepest descent
	(SD) method \cite[Chapter 5]{Saad2003} as a simple instance.
	
	For arbitrary given tensors
	$\mathscr{X},\mathscr{Y}\in \mathbb{C}^{n\times 1 \times p}$, Kilmer et
	al. \cite{Kilmer2013} defined the bilinear form
	$\left\langle {\mathscr{X},\mathscr{Y}} \right\rangle:=\mathscr{X}^H \ast
	\mathscr{Y}$ which can be regarded as a generalization for the notion of inner
	product. The following lemma presents the properties of bilinear form
	$\left\langle {\cdot,\cdot} \right\rangle$.
	
	\begin{lem} {\rm \cite[Lemma 3.1]{Kilmer2013}}
		Let $\mathscr{X},\mathscr{Y},\mathscr{Z}\in \mathbb{C}^{n\times 1 \times p}$ and 
		$[a]$ be a tubular tensor of length $p$. Then $\left\langle {\mathscr{X},\mathscr{Y}} \right\rangle$ satisfies the following properties:
		\begin{itemize}
			\item  $\left\langle {\mathscr{X},\mathscr{Y}+\mathscr{Z}} \right\rangle=\left\langle {\mathscr{X},\mathscr{Y}} \right\rangle+\left\langle {\mathscr{X},\mathscr{Z}} \right\rangle$
			\item  $\left\langle {\mathscr{X},\mathscr{Y}\ast [a]} \right\rangle=[a]\ast  \left\langle {\mathscr{X},\mathscr{Y}} \right\rangle$
			\item  $\left\langle {\mathscr{X},\mathscr{Y}} \right\rangle=\left\langle {\mathscr{Y},\mathscr{X}} \right\rangle^H$.
		\end{itemize} 
	\end{lem}
	
	\noindent Given $\mathscr{X}\in \mathbb{C}^{n\times 1 \times p}$, the
	\textit{tubular norm} of $\mathscr{X}$ can be also defined by
	$\|\mathscr{X}\|:=\left\langle {\mathscr{X},\mathscr{X}}
	\right\rangle^{\frac{1}{2}}$. The tubular tensor
	$\|\mathscr{X}\|\ast \|\mathscr{X}\|$ is denoted by $ \|\mathscr{X}\|^2$ for
	simplicity.  As can be seen, the tubular norm is a map form
	$\mathbb{C}^{n\times 1 \times p}$ to the set of symmetric (semi-)positive definite tubular
	tensors of length $p$.  In the following, Lemma \ref{lem5.5} reveals that the
	map $\|\cdot\|$ has properties similar to the standard scalar norm. To show 
	the lemma, we only need to present the Cauchy--Schwarz and triangle inequalities
	with respect to the bilinear form $\left\langle {\cdot,\cdot} \right\rangle$
	which are respectively given in Lemma \ref{lem:CS} and Remark \ref{tineq}.

	\begin{lem} \label{lem:CS}
		Let $\mathscr{X},\mathscr{Y}\in \mathbb{C}^{n\times 1 \times p}$. Then the following relation holds 
		\[
		- 	\left\langle {\mathscr{X},\mathscr{X}} \right\rangle^{\frac{1}{2}} \ast 	\left\langle {\mathscr{Y},\mathscr{Y}} \right\rangle^{\frac{1}{2}} \preceq \frac{1}{2}\left(\left\langle {\mathscr{X},\mathscr{Y}} \right\rangle + \left\langle {\mathscr{X},\mathscr{Y}} \right\rangle^H \right)\preceq	\left\langle {\mathscr{X},\mathscr{X}} \right\rangle^{\frac{1}{2}} \ast 	\left\langle {\mathscr{Y},\mathscr{Y}} \right\rangle^{\frac{1}{2}}.
		\] 	
	\end{lem}
	
	\begin{proof} It is known that 
		\[
		\bcrc(\mathscr{X})= (F_p \otimes I_n) {\rm blockdiag}(\tilde{x}_1,\tilde{x}_2,\ldots,\tilde{x}_p)F_p^H \quad \text{and} \quad \bcrc(\mathscr{Y})= (F_p \otimes I_n) {\rm blockdiag}(\tilde{y}_1,\tilde{y}_2,\ldots,\tilde{y}_p)F_p^H.
		\]
		The Cauchy--Schwarz inequality implies that $|\tilde{x}_i^H\tilde{y}_i| \le \|\tilde{x}_i\|_2 \|\tilde{y}_i\|_2$ for $i=1,2,\ldots,p$ which implies that
		\begin{eqnarray*}
			- 2 F_p \text{diag}( \|\tilde{x}_1\|_2 \|\tilde{y}_1\|_2,\ldots, \|\tilde{x}_p\|_2 \|\tilde{y}_p\|_2) F_p^H	& \preceq & \bcrc(\mathscr{X})^H\bcrc(\mathscr{Y})+\bcrc(\mathscr{Y})^H\bcrc(\mathscr{X})\\
			&=&F_p \text{diag}(2\Re(\tilde{x}_1^H\tilde{y}_1),\ldots,2\Re(\tilde{x}_p^H\tilde{y}_p))F_p^H \\
			& \preceq & 2 F_p \text{diag}( \|\tilde{x}_1\|_2 \|\tilde{y}_1\|_2,\ldots, \|\tilde{x}_p\|_2 \|\tilde{y}_p\|_2) F_p^H
		\end{eqnarray*}
		where $\Re(z)$ denotes the real part of an arbitrary complex number $z$. The proof follows immediately  from the above relations.
	\end{proof}
	
	\begin{rem}\label{tineq}
		Let $\mathscr{X},\mathscr{Y}\in \mathbb{C}^{n\times 1 \times p}$. 	Evidently, we have 
		$
		\left\langle {\mathscr{X}+\mathscr{Y},\mathscr{X}+\mathscr{Y}} \right\rangle=\left\langle {\mathscr{X},\mathscr{X}} \right\rangle+\left\langle {\mathscr{X},\mathscr{Y}} \right\rangle+\left\langle {\mathscr{X},\mathscr{Y}} \right\rangle^H+\left\langle {\mathscr{Y},\mathscr{Y}} \right\rangle.
		$
		From Lemma \ref{lem:CS}, we conclude that
		\begin{eqnarray*}
			\left\langle {\mathscr{X}+\mathscr{Y},\mathscr{X}+\mathscr{Y}} \right\rangle &\preceq&  \left\langle {\mathscr{X},\mathscr{X}} \right\rangle +\left\langle {\mathscr{Y},\mathscr{Y}} \right\rangle + 2 	\left\langle {\mathscr{X},\mathscr{X}} \right\rangle^{\frac{1}{2}} \ast 	\left\langle {\mathscr{Y},\mathscr{Y}} \right\rangle^{\frac{1}{2}} \\
			&=& (\left\langle {\mathscr{X},\mathscr{X}} \right\rangle^{\frac{1}{2}}+	\left\langle {\mathscr{Y},\mathscr{Y}} \right\rangle^{\frac{1}{2}}) \ast (\left\langle {\mathscr{X},\mathscr{X}} \right\rangle^{\frac{1}{2}} + 	\left\langle {\mathscr{Y},\mathscr{Y}} \right\rangle^{\frac{1}{2}}).
		\end{eqnarray*}
		Hence, we deduce
		\[
		\left\langle {\mathscr{X}+\mathscr{Y},\mathscr{X}+\mathscr{Y}} \right\rangle^{\frac{1}{2}} \preceq \left\langle {\mathscr{X},\mathscr{X}} \right\rangle^{\frac{1}{2}}  + \left\langle {\mathscr{Y},\mathscr{Y}} \right\rangle^{\frac{1}{2}}.	
		\]
		Basically, the preceding relation can be seen as an extension of the triangle inequality with respect to the bilinear form $\left\langle {\cdot,\cdot} \right\rangle$.
	\end{rem}

	\begin{lem}\label{lem5.5}
		Let $\mathscr{X},\mathscr{Y}\in \mathbb{C}^{n\times 1 \times p}$. The following properties hold:
		\begin{itemize}
			\item $\|\mathscr{X}\| \succeq 0$ and $\|\mathscr{X}\|$ is a zero tubular tensor iff $\mathscr{X}$ is a zero tensor, i.e., $\mathscr{X}=\mathscr{O}$.
			\item $\|k\mathscr{X} \|=|k|\|\mathscr{X} \|$ for any complex number $k$.
			\item $\|\mathscr{X} \ast [a] \|=([a]^H\ast [a])^{\frac{1}{2}}\ast \|\mathscr{X} \|$
			for any tubular tensor $[a]$ of length $p$.
			\item $\| \mathscr{X}  + \mathscr{Y} \| \preceq \| \mathscr{X}  \| + \|  \mathscr{Y} \|.$
		\end{itemize} 
	\end{lem}

	Now, we compare the accuracy of the approximations produced by tubular and
	global iterative methods. Let $\mathscr{X}_{old}$ be an available approximation
	to the solution of \eqref{normal}.  The computed new approximations can be
	regarded in the form $\mathscr{X}_{old}+\mathcal{S}$ where the tensor
	$\mathcal{S}$ belongs to the same subspace in both tubular and global iterative
	methods.  To be more precise, suppose that
	$\mathscr{V}_1,\mathscr{V}_2,\ldots,\mathscr{V}_m$ are T-linearly independent
	$n\times 1 \times p$ tensors. It follows immediately that
	$\ufld(\mathscr{V}_1),\ufld(\mathscr{V}_2),\ldots,\ufld(\mathscr{V}_m)$ are a
	linearly independent set of vectors.  The new approximation in tubular and
	global iterative methods are respectively obtained as follows:
	\begin{eqnarray}
		\tilde{\mathscr{X}}_{new} & = &  \mathscr{X}_{old}+ \sum_{\ell=1}^{m} \mathscr{V}_{\ell} \ast [\alpha_{\ell}]  \label{tubal}\\
		\bar{\mathscr{X}}_{new} & = &  \mathscr{X}_{old}+ \sum_{\ell=1}^{m} \mu_{\ell} \mathscr{V}_{\ell} \label{global}
	\end{eqnarray}
	such that the tubular tensors $[\alpha_1], [\alpha_2], \ldots, [\alpha_m]$ are determined by imposing the following $m$ sets of restrictions:
	\[
	\left\langle {\tilde{\mathscr{R}}_{new},\mathscr{V}_\ell} \right\rangle =[0_{p \times 1}], \quad \text{for} \quad  \ell=1,2,\ldots,m
	\]
	and the parameters $\mu_1,\mu_2,\ldots,\mu_k$ are computed using the following orthogonality conditions
	\[
	\left\langle {\bar{\mathscr{R}}_{new},\mathscr{V}_\ell} \right\rangle_F=0 \quad \text{for} \quad \ell=1,2,\ldots,m
	\]
	where $\tilde{\mathscr{R}}_{new}=\mathscr{A}^T\ast \mathscr{B}-\mathscr{A}^T\ast \mathscr{A}\ast \tilde{\mathscr{X}}_{new}$,  $\bar{\mathscr{R}}_{new}=\mathscr{A}^T\ast\mathscr{B}-\mathscr{A}^T\ast\mathscr{A}\ast \bar{\mathscr{X}}_{new}$ and $0_{p \times 1}$ is the zero vector of size $p$. Notice that the iterative method \eqref{global} is mathematically equivalent to the following form
	\[
	\bar{\mathscr{X}}_{new} = \mathscr{X}_{old}+ \sum_{\ell=1}^{m}  \mathscr{V}_{\ell} \ast [\bar{\alpha}_\ell]
	\]
	where $[\bar{\alpha}_\ell]=\mu_\ell [e_1]$ for $\ell=1,2,\ldots,m$. The
	following theorem shows that the approximate solution obtained by iterative
	method \eqref{tubal} satisfies an optimality condition.

	\begin{thm}\label{opt}
		Let the tensor $\tilde{\mathscr{X}}_{new}$ be defined as above  and
		let $\mathscr{X}^*$ be the exact solution of $\mathscr{A}\ast \mathscr{X}=\mathscr{B}$,
		where it is  assumed that  $\mathscr{A}^T \ast \mathscr{A}$ is non-singular.
		Then,
		\begin{equation}
			\left\langle {(\mathscr{A}^T \ast \mathscr{A})\ast(\tilde{\mathscr{X}}_{new}-\mathscr{X}^*),(\tilde{\mathscr{X}}_{new}-\mathscr{X}^*) } \right\rangle \prec \left\langle {(\mathscr{A}^T \ast \mathscr{A})\ast(\mathscr{X}-\mathscr{X}^*),(\mathscr{X}-\mathscr{X}^*) } \right\rangle
		\end{equation}
		for any $\mathscr{X} \ne \tilde{\mathscr{X}}_{new}$ of the form $\mathscr{X}_{old}+\mathscr{S}$ where $\mathscr{S}$ is a (T)-linear combination of the T-linearly independent tensors $\mathscr{V}_1,\mathscr{V}_2,\ldots,\mathscr{V}_m$ where $m\ge 1$ is a given integer.
	\end{thm}
	\begin{proof}
		For notational simplicity, we set $\tilde{\mathscr{E}}:=\tilde{\mathscr{X}}_{new}-\mathscr{X}^*$ and $\mathscr{E}:=\mathscr{X}-\mathscr{X}^*$. By straightforward computations, it turns out that
		\[
		\left\langle {(\mathscr{A}^T \ast \mathscr{A})\ast \tilde{\mathscr{E}},\tilde{\mathscr{E}}} \right\rangle = \left\langle {(\mathscr{A}^T \ast \mathscr{A})\ast \mathscr{E},\mathscr{E}} \right\rangle - \left\langle {(\mathscr{A}^T \ast \mathscr{A})\ast (\tilde{\mathscr{X}}_{new}-\mathscr{X}), (\tilde{\mathscr{X}}_{new}-\mathscr{X})} \right\rangle
		\]
		which completes the proof.
	\end{proof}
	
	Let the tubular tensors $[a]$ and $[b]$ be given. If $[a]\ast [a] \preceq  [b] \ast [b]$, then $([a]\ast [a])_{::1} \preceq  ([b] \ast [b])_{::1}$ which is equivalent to  $\|[a]\|^2_F \le \|[b]\|^2_F$. Consequently, Theorem \ref{opt} reveals that\footnote{The tensor $(\mathscr{A}^T \ast \mathscr{A})^{1/2}$,  as the square root of $\mathscr{A}^T \ast \mathscr{A}$,  is well-defined by the assumptions of Theorem \ref{opt}.}
	$$
	\|(\mathscr{A}^T \ast \mathscr{A})^{1/2}\ast  (\tilde{\mathscr{X}}_{new}-\mathscr{X}^*)\|_F \le \| (\mathscr{A}^T \ast \mathscr{A})^{1/2}\ast (\bar{\mathscr{X}}_{new}-\mathscr{X}^*) \|_F
	$$
	which shows that the approximate solution computed by \eqref{tubal} is more
	accurate than the one obtained by \eqref{global}.
	To give an instance for the
	role played by the tubular eigenvalues of $\mathscr{A}^T \ast \mathscr{A}$
	in convergence analysis 
	\footnote{In the
		case when $\mathscr{A}$ is symmetric positive definite, the mentioned
		iterative methods in this section are directly applied for
		$\mathscr{A} \ast \mathscr{X}=\mathscr{B}$ and the convergence results are
		changed accordingly.}
	of tubular orthogonal projection
	methods for solving \eqref{normal}, in the sequel, we limit our discussion to a
	simple case where $m=1$ in \eqref{tubal}, i.e., the tubular version of SD
	method.
	
	Let $\mathscr{D}_k=\mathscr{A}^T \ast (\mathscr{B}-\mathscr{A} \ast \mathscr{X}_{k})$ where $\mathscr{X}_{k}$ is the $k$th approximate solution to \eqref{eqten}, the  algorithm can be applied with respect to the bilinear form $\left\langle {\cdot,\cdot} \right\rangle$ in the following manner,
	\begin{equation}\label{SD1}
		\mathscr{X}_{k+1}=\mathscr{X}_{k}+\mathscr{D}_k \ast  \|\mathscr{D}_k\|^2 \ast (\|\mathscr{A} \ast \mathscr{D}_k\|^2)^{-1}, \quad k=0,1,2,\ldots
	\end{equation}
	Notice that the global SD algorithm for solving \eqref{normal} is implemented as follows: 
	\begin{equation}\label{SD2}
		\mathscr{X}_{k+1}=\mathscr{X}_{k}+\frac{\|\mathscr{D}_k\|_F^2}{\|\mathscr{A} \ast \mathscr{D}_k\|_F^2} ~\mathscr{D}_k, \quad k=0,1,2,\ldots.
	\end{equation}

	Considering the ordering \eqref{ordered:eig} and using the strategy exploited in \cite[Lemma 5.8]{Saad2003}, the well-known Kantorovich inequality can be generalized with respect to bilinear form $\left\langle {\cdot,\cdot} \right\rangle$ as follows:
	
	\begin{lem}\label{Kantorovich}
		Let $\mathscr{A}$ be an $n\times n \times p$ Hermitian positive definite tensor and  $[\lambda_m]$ and $[\lambda_M]$ be two extreme tubular eigenvalues of $\mathscr{A}$ defined as in \eqref{ordered:eig}. Let $\mathscr{X}$ be an arbitrary $n\times 1 \times p$ tensor such that $[\mathscr{X}^H \ast \mathscr{X}]$ is non-singular, then
		\begin{equation}\label{Kant}
			(\mathscr{X}^H \ast \mathscr{X})^{-1}\ast	(\mathscr{X}^H \ast \mathscr{A} \ast \mathscr{X})\ast(\mathscr{X}^H \ast \mathscr{A}^{-1} \ast \mathscr{X})\ast	(\mathscr{X}^H \ast \mathscr{X})^{-1}\preceq \frac{1}{4}([\lambda_m]+ [\lambda_M])^2\ast[\lambda_m^{-1}]\ast[\lambda_M^{-1}].
		\end{equation}
	\end{lem} 
	\begin{proof}
		Consider the decomposition \eqref{udiag} where 
		\[
		\bcrc(\mathscr{D})=	(F_p\otimes I_n){\rm blockdiag}(\tilde{D}_1,\tilde{D}_2,\ldots,\tilde{D}_p)(F_p^H\otimes I_n)
		\]
		in which	$\tilde{D}_i={\rm diag}(\tilde{d}_1^{(i)},\tilde{d}_2^{(i)},\ldots,\tilde{d}_n^{(i)})$
		for $i=1,2,\ldots,p$ such that \eqref{ordered:diag} holds.
		Without loss of generality, we may assume that $\mathscr{X}^H \ast \mathscr{X}=[e_1]$.
		It can be observed that
		\[
		\mathscr{X}^H \ast \mathscr{A} \ast \mathscr{X}=	\mathscr{Y}^H \ast \mathscr{D} \ast \mathscr{Y}
		\quad
		\text{and}
		\quad
		\mathscr{X}^H \ast \mathscr{A}^{-1} \ast \mathscr{X}=\mathscr{Y}^H \ast \mathscr{D}^{-1} \ast \mathscr{Y}
		\]
		where $\mathscr{Y}=\mathscr{Q}\ast \mathscr{X}$. Let 
		$
		\bcrc(\mathscr{Y})=(F_p\otimes I_n) {\rm blockdiag}(\tilde{y}_1,\tilde{y}_2,\ldots,\tilde{y}_p)F_p^H.
		$
		Evidently, we have $\mathscr{Y}^H \ast \mathscr{Y}=[e_1]$ which implies $\tilde{y}_i^H\tilde{y}_i=1$  for $i=1,2,\ldots,p$.
		A straightforward computation shows that
		\[
		\crc(\mathscr{Y}^H \ast \mathscr{D} \ast \mathscr{Y})=F_p\text{diag}(\tilde{y}_1^H\tilde{D}_1\tilde{y}_1,\tilde{y}_2^H\tilde{D}_2\tilde{y}_2,\ldots,\tilde{y}_p^H\tilde{D}_p\tilde{y}_p)F_p^H
		\]
		and
		\[
		\crc(\mathscr{Y}^H \ast \mathscr{D}^{-1}\ \ast \mathscr{Y})=F_p\text{diag}(\tilde{y}_1^H\tilde{D}_1^{-1}\tilde{y}_1,\tilde{y}_2^H\tilde{D}_2^{-1}\tilde{y}_2,\ldots,\tilde{y}_p^H\tilde{D}_p^{-1}\tilde{y}_p)F_p^H.
		\]
		Similar to the proof of \cite[Lemma 5.8]{Saad2003}, we can deduce that 
		\[
		\left( \tilde{y}_j^H\tilde{D}_j\tilde{y}_j \right)\left( \tilde{y}_j^H\tilde{D}_j^{-1}\tilde{y}_j \right) \le \frac{(\tilde{d}_1^{(j)}+\tilde{d}_n^{(j)})^2}{4\tilde{d}_1^{(j)}\tilde{d}_n^{(j)}} \quad \text{for} \quad  j=1,2,\ldots,p.
		\]
		Consequently, we have
		\[
		\crc(\mathscr{Y}^H \ast \mathscr{D} \ast \mathscr{Y}) \crc(\mathscr{Y}^H \ast \mathscr{D}^{-1}\ \ast \mathscr{Y}) \preceq F_P \text{diag}\left(\frac{(\tilde{d}_1^{(1)}+\tilde{d}_n^{(1)})^2}{4\tilde{d}_1^{(1)}\tilde{d}_n^{(1)}},\frac{(\tilde{d}_1^{(2)}+\tilde{d}_n^{(2)})^2}{4\tilde{d}_1^{(2)}\tilde{d}_n^{(2)}},\ldots ,\frac{(\tilde{d}_1^{(p)}+\tilde{d}_n^{(p)})^2}{4\tilde{d}_1^{(p)}\tilde{d}_n^{(p)}}\right)F_p^H
		\]
		which completes the proof.
	\end{proof}

	As a result, one can prove that the iterative method \eqref{SD1} is convergent by using  Lemma \ref{lem2.4} and the following theorem. The proof of theorem follows from straightforward computations and the analogous strategies used in \cite[Theorem 5.9]{Saad2003}.
	
	\begin{thm}\label{SDcon}
		Let  $\mathscr{A}$ be nonsingular and $ \mathscr{X}^*$ be the unique  solution of  $\mathscr{A}^T\ast \mathscr{A} \ast \mathscr{X} =\mathscr{A}^T \ast \mathscr{B}$.
		Assume that $\mathscr{E}_k=\mathscr{X}^*-\mathscr{X}_k$ where $\mathscr{X}_k$ stands for the $k$th approximate solution obtained by \eqref{SD1}. Then,
		\[
		[\left\langle { \mathscr{A} \ast{\mathscr{E}_{k+1}},\mathscr{A}\ast {\mathscr{E}_{k+1}}} \right\rangle] \preceq [w]  \ast [\left\langle {\mathscr{A} \ast{\mathscr{E}_k},\mathscr{A}\ast {\mathscr{E}_k}} \right\rangle]
		\]
		where 
		$
		[w] = \left([k]-[e_1]\right)^{2} \ast \left([k]+[e_1]\right)^{-2}
		$
		with $[k]:=[\lambda_M(\mathscr{A}^T \ast \mathscr{A})]\ast[(\lambda_m(\mathscr{A}^T \ast \mathscr{A}))^{-1}]$.\footnote{In the proof of Theorem \ref{SDcon}, one need to use that fact that $[a]\preceq [b]$ impels $[b^{-1}]\preceq [a^{-1}]$ for arbitrary given symmetric positive tubular tensors $a$ and $b$.}
	\end{thm}
	
	In view of the above theorem, we can observe that the convergence of iterative method \eqref{SD1} deteriorates
	when the matrices $\tilde{A}_1^H\tilde{A}_1,\tilde{A}_2^H\tilde{A}_2\ldots,\tilde{A}_p^H\tilde{A}_p$ are too ill-conditioned where $\tilde{A}_1,\tilde{A}_2\ldots,\tilde{A}_p$ are given in \eqref{bloc:diag}. 
	
	\subsection{Numerical experiments}\label{num} In the sequel, we report on some
	experimental results on the performances of tubular and global versions of
	mentioned iterative methods to solve \eqref{normal}. All of the numerical
	computations were carried out on a computer with an Intel Core i7-10750H CPU @
	2.60GHz processor and 16.0GB RAM using MATLAB.R2020b. The tensor $\mathscr{B}$
	is generated so that $\mathscr{X}^*=\verb|randn|(n,1,n)$ is the exact solution
	of \eqref{eqten}.  The performance of iterative methods \eqref{richard1} and
	\eqref{richard2} were respectively examined for $[\alpha]=[\alpha^*],[\alpha_1]$
	and $\mu=\mu^*,\mu_1$ where the tubular tensor $[\alpha^*]$ (scalar $\mu^*$) is
	defined by Eq. \eqref{tubopt} (Eq. \eqref{richardopt}),
	\[
	[\alpha_1] = [\lambda_M(\mathscr{A}^T \ast\mathscr{A})^{-1}] \quad \text{and} \quad \mu_1 =\bar{\lambda}_M^{-1}
	\]
	here $\bar{\lambda}_M$ denotes the maximum T-eigenvalue of  $\mathscr{A}^T \ast \mathscr{A}$.
	For simplicity, we use the abbreviated names in the figures which are listed in Table \ref{abname}.
	In the figures, corresponding to each method, number of iterations ($k$) are displayed versus ${\rm Log}_{10}\delta_k$ where
	\begin{equation*}
		\delta_k:=\frac{\|{\mathscr{B}}-\mathscr{A} \ast \mathscr{X}_{k}\|_F}{\|{\mathscr{B}}\|_F}
	\end{equation*}
	here $\mathscr{X}_{k}$ ($k\ge 1$) is the $k$th computed approximate solution and $\mathscr{X}_{0}$  is taken to be zero.
	
	\begin{table}[ht]
		\centering
		{\scriptsize \caption{The list of abbreviated names for methods}\label{abname}}
		\begin{tabular}{lccc}
			\hline
			Method &&& Abbreviated name \\
			\hline
			Global version of Richardson &&& Richardson\\
			Global version of Steepest Descent   &&& SD\\
			Tubular Richardson &&& TR \\
			TR with relaxation step &&& TRR\\
			Tubular SD &&& TSD\\
			TSD with relaxation step &&& TSDR\\
			\hline		
		\end{tabular}	
	\end{table}

	\begin{example}\label{ex1}\rm \cite[Example 6.3]{Reichel} 
		Let $\mathscr{A}$ in \eqref{eqten} be an $256 \times 256 \times 256$ tensor such that
		$\mathscr{A}_{::i}=A(i,1)A$ for $i=1,2\ldots, n$ in which
		\[	A=\frac{1}{{\sqrt{2\pi\sigma^2}}}\cdot\verb|toeplitz|\left(\verb|[z(1) fliplr(z(2:end))]|,\verb|z|\right)
		\]
		with
		$\verb|z|=\left[{\verb|exp|\left(-([0:{\rm \verb|band| }-1]_{\cdot}^2)/{(2 \sigma^2)}\right),\verb|zeros|(1,256-\verb|band|)} \right]$
		where $\verb|band|=7$ and $\sigma=4$.
	\end{example}
	
	The  tensor equation in Example \ref{ex1} is essentially a well-conditioned problem. So, as anticipated, the TR (with $[\alpha^*]$ and $[\alpha_1]$), Richardson (with $\mu^*$ and $\mu_1$), TSD and SD methods work fine. Basically, our numerical observations confirm the fact that the tubular versions of the method outperform their global ones. For further details, we plot the convergence history of the proposed methods in Figure \ref{fig11}.  
	
	\begin{figure}[H]
		\centering
		\includegraphics[width=.7\linewidth]{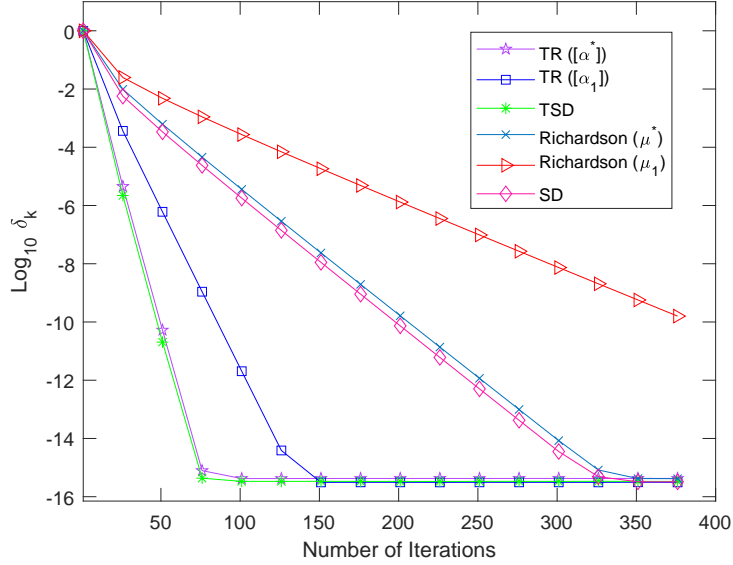}	
		\caption{Convergence history of examined iterative methods for Example \ref{ex1}.  
		}
		\label{fig11}
	\end{figure}

	\begin{example}\label{ex2}\rm \cite[Example 6.1]{Reichel} 
		Let $\mathscr{A}$ in \eqref{eqten} be an $n \times n \times n$ tensor such that
		$\mathscr{A}_{::i}=A_1(i,1)A_2$ for $i=1,2\ldots, n$  in which the matrices $A_1$ and $A_2$ are respectively generated by the function $\verb|baart|(n)$ in Hansen's package \cite{Hansen2007} and MATLAB function $\verb|gallery|('\verb|prolate|',n, 0.46)$.
	\end{example}
	
	As mentioned in \cite{Reichel}, the coefficient tensor in the above example is
	severely ill-conditioned. Hence, the SD and TSD methods do not work well. Our
	numerical experiments also illustrate that the Richardson and TR (with
	$[\alpha^*]$) methods stagnate.  However, numerical experiments show that the TR
	method with $[\alpha_1]$ can act as an iterative regularization method. To
	monitor the accuracy of obtained approximation corresponding to $\mathscr{X}^*$,
	the TR method (with $[\alpha_1]$) was used with a relative residual tolerance of
	$5\cdot 10^{-3}$. In this case, in average, the method stopped after about $17$ iterations and
	produced an approximate solution with relative error equal to $0.2601$ for
	$n=256$. In \cite{Reichel}, the right-hand side tensor $\mathscr{B}$ is
	generated such that $\mathscr{B}=\mathscr{A} \ast \hat {\mathscr{X}}^*$ with
	$ \hat {\mathscr{X}}^*={\rm ones}(256,1,256)$. This is a special case for which
	the TR method (with $[\alpha_1]$) requires less than five iterations with
	respect to our exploited stopping criterion, e.g., in a specific run it results
	an approximate solution with the relative error $0.0094$ after three iterations.
	
	To improve performance of the TR method (with $[\alpha_1]$) for solving Example
	\ref{ex2}, we can combine it with some kind of relaxation step. For instance,
	we experimentally observed that performance of TR method (with $[\alpha_1]$) can
	be improved by using the following simple relaxation step which is originally
	proposed in \cite{Antuono} for accelerating the convergence speed of iterative
	methods to solve linear system of equations. To apply the relaxation step,
	first, let us consider the mentioned iterative methods in the following form:
	\[
	\mathscr{X}_{k+1}=\mathscr{X}_k+\mathcal{F}(\mathscr{X}_k),\quad  k=0,1,2,\ldots.
	\]
	After performing two steps, we can add relation to  the method as follows:
	\begin{eqnarray*}
		\bar{\mathscr{X}}_{k+1}&=&\mathscr{X}_k+\mathcal{F}(\mathscr{X}_k)\\
		{\mathscr{X}}_{k+1}&=&	{\mathscr{X}}_{k-1}+  \omega_{k}  	(\bar{\mathscr{X}}_{k+1}-	\bar{\mathscr{X}}_{k-1})
	\end{eqnarray*}
	where, in the experimental results shown, the scalar $ \omega_k$ is determined
	by minimum residual technique with respect to scalar norm of tensor at each step
	\[\omega_k  = \frac{{{{\left\langle {{\mathscr{R}_{k - 1}},{\mathscr{R}_{k - 1}} - {\bar{\mathscr{R}}_{k + 1}}} \right\rangle }_F}}}{{\left\| {{\mathscr{R}_{k - 1}} - \bar{{\mathscr{R}}}_{k + 1}} \right\|^2_F}}\]
	here $\mathscr{R}_{k - 1}=\mathscr{B}-\mathscr{A} \ast \mathscr{X}_{k - 1}$ and
	$\bar{{\mathscr{R}}}_{k + 1}=\mathscr{B}-\mathscr{A} \ast \bar{{\mathscr{X}}}_{k
		+ 1}$\footnote{We comment that $\omega_{k}$ may be chosen to be a constant
		tubular tensor and the relaxation method can be also used with other possible
		approaches. Although these kinds of changes may lead to better results, we do
		not consider them in this work. Basically, here, the main goal is to highlight the
		role of tubular eigenvalue analysis in studying the convergence of  tubular iterative methods
		and possibly developing suitable preconditioners for \eqref{eqten} in the
		future works.}. In Figure \ref{fig2}, for the sake of comparison, the
	performance of TR and TRR methods (with $[\alpha_1]$) are displayed.  For TSD
	and SD methods, based on our observations, the above utilized relaxation step
	may not be successful due to the fact that the problem is too
	ill-conditioned. Specifically, the sequence of numerical approximation generated
	by TSD and SD methods even with relaxation step diverge when $n> 100$ in Example
	\ref{ex2}. To demonstrate the effect of relaxation step, the convergence history
	of the methods are plotted in Figure \ref{fig22} for $n=100$.

	%

	\begin{figure}[h]
		\centering
		\includegraphics[width=.7\linewidth]{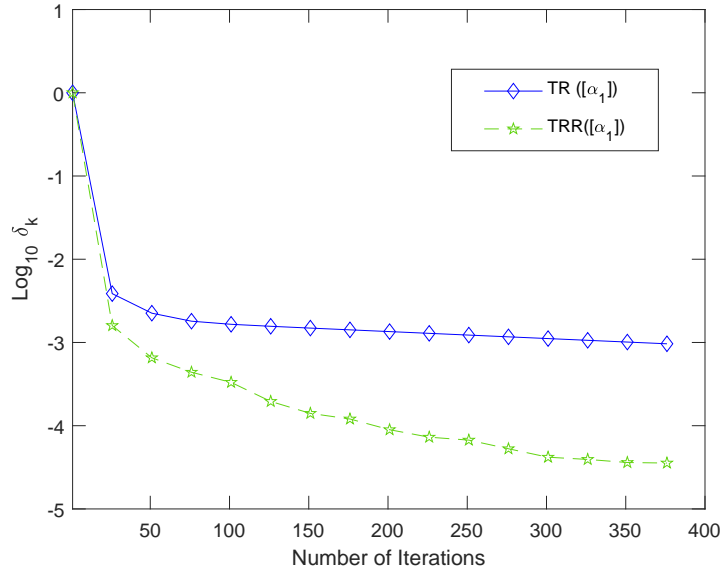} 	
		\caption{Convergence history of examined iterative methods in Example \ref{ex2} for $n=256$.  
		}
		\label{fig2}
	\end{figure}
	
	\begin{figure}[h]
		\centering
		\includegraphics[width=.7\linewidth]{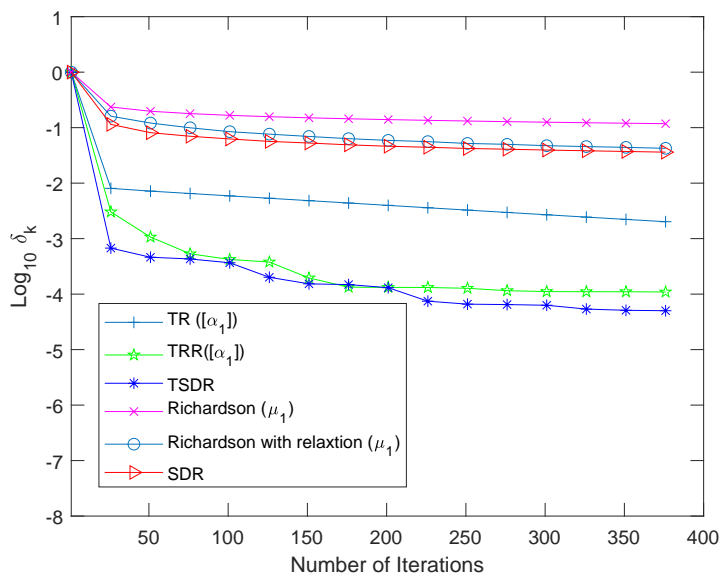}
		\caption{Convergence history of examined iterative methods in Example \ref{ex2} for $n=100$.  
		}
		\label{fig22}
	\end{figure}

	\section{Conclusion}\label{sec5}
	
	We analyzed the properties of eigenvalues of third-order tensors with respect to
	T-product. More precisely, the eigenvalues of tensors were considered as tubular
	tensors. Links were established between this type of eigenvalues with
	T-eigenvalues and eigentuples which are two alternative definitions of
	eigenvalues for tensors with respect to the T-product. Some results were also
	included to indicate the role of tubular eigenvalues in convergence analysis of
	tubular iterative methods for solving tensor equation in the form
	$\mathscr{A} \ast \mathscr{X} = \mathscr{B}$. In addition, we briefly mentioned
	that the tubular form of iterative methods surpasses their
	alternative tensor version which is mathematically equivalent to applying them
	for solving $\bcrc(\mathscr{A}) \ufld(\mathscr{X})=\ufld(\mathscr{B})$.
	
	Future work that can potentially benefit from the notion of tubular eigenvalues
	could consider research for speeding up the convergence of tubular Krylov
	subspace methods to solve tensor equations with respect to T-product by
	proposing suitable preconditioners or using other techniques such as deflation
	and augmentation.
	
	\section*{Disclosure statement}
	The authors declare no potential conflict of interest.
	
	\bibliographystyle{siam}
	\bibliography{refs}

\end{document}

%% file: macros.tex
\newtheorem{thm}{Theorem}[section]

\newtheorem{lem}[thm]{Lemma}
\newtheorem{prop}[thm]{Proposition}
\newtheorem{defn}[thm]{Definition}
\newtheorem{rem}[thm]{\bf{Remark}}
\newtheorem{example}[thm]{\bf{Example}}

\definecolor{fgreen}{rgb}{0.13, 0.55, 0.13}

\newcommand{\C}{\mathbb C}

\newcommand{\cX}{\mathscr{X}}
\newcommand{\cV}{\mathscr{V}}

\newcommand{\cA}{\mathscr{A}}

\newcommand{\eq}[1]{\begin{equation}\label{#1}}
\newcommand{\en}{\end{equation}}
\newcommand{\ufld}{\textsl{Ufold}}
\newcommand{\fld}{\textsl{Fold}}
\newcommand{\crc}{{\rm circ}}
\newcommand{\bcrc}{{\rm bcirc}}

%% file: paper_1.bbl
\begin{thebibliography}{10}

\bibitem{Antuono}
{\sc M.~Antuono and G.~Colicchio}, {\em Delayed over-relaxation for iterative
  methods}, Journal of Computational Physics, 321 (2016), pp.~892--907.

\bibitem{Beik}
{\sc F.~P.~A. Beik, A.~E. Ichi, K.~Jbilou, and R.~Sadaka}, {\em Tensor
  extrapolation methods with applications}, Numerical Algorithms, 87 (2021),
  pp.~1421--1444.

\bibitem{Braman}
{\sc K.~Braman}, {\em Third-order tensors as linear operators on a space of
  matrices}, Linear Algebra and its Applications, 433 (2010), pp.~1241--1253.

\bibitem{Chan}
{\sc R.~H.-F. Chan and X.-Q. Jin}, {\em An introduction to iterative Toeplitz
  solvers}, SIAM, 2007.

\bibitem{El2021tensor}
{\sc M.~El~Guide, A.~El~Ichi, K.~Jbilou, and R.~Sadaka}, {\em On tensor
  \uppercase{GMRES} and \uppercase{G}olub-\uppercase{K}ahan methods via the
  \uppercase{T}-product for color image processing}, Electron. J. Linear
  Algebra, 37 (2021), pp.~524--543.

\bibitem{ElIchi}
{\sc A.~El~Ichi, K.~Jbilou, and R.~Sadaka}, {\em On tensor
  tubal-\uppercase{K}rylov subspace methods}, Linear and Multilinear Algebra,
  (2021), pp.~1--24.

\bibitem{Hansen2007}
{\sc P.~C. Hansen}, {\em Regularization tools version 4.0 for matlab 7.3},
  Numerical algorithms, 46 (2007), pp.~189--194.

\bibitem{Hao}
{\sc N.~Hao, M.~E. Kilmer, K.~Braman, and R.~C. Hoover}, {\em Facial
  recognition using tensor-tensor decompositions}, SIAM Journal on Imaging
  Sciences, 6 (2013), pp.~437--463.

\bibitem{Khaleel}
{\sc H.~S. Khaleel, S.~V.~M. Sagheer, M.~Baburaj, and S.~N. George}, {\em
  Denoising of \uppercase{R}ician corrupted 3\uppercase{D} magnetic resonance
  images using tensor-\uppercase{SVD}}, Biomedical Signal Processing and
  Control, 44 (2018), pp.~82--95.

\bibitem{Kilmer2013}
{\sc M.~E. Kilmer, K.~Braman, N.~Hao, and R.~C. Hoover}, {\em Third-order
  tensors as operators on matrices: \uppercase{A} theoretical and computational
  framework with applications in imaging}, SIAM Journal on Matrix Analysis and
  Applications, 34 (2013), pp.~148--172.

\bibitem{Kilmer2021}
{\sc M.~E. Kilmer, L.~Horesh, H.~Avron, and E.~Newman}, {\em Tensor-tensor
  algebra for optimal representation and compression of multiway data},
  Proceedings of the National Academy of Sciences, 118 (2021), p.~e2015851118.

\bibitem{Kilmer2011}
{\sc M.~E. Kilmer and C.~D. Martin}, {\em Factorization strategies for
  third-order tensors}, Linear Algebra and its Applications, 435 (2011),
  pp.~641--658.

\bibitem{Kolda}
{\sc T.~G. Kolda and B.~W. Bader}, {\em Tensor decompositions and
  applications}, SIAM review, 51 (2009), pp.~455--500.

\bibitem{Liu}
{\sc W.-h. Liu and X.-q. Jin}, {\em A study on \uppercase{T}-eigenvalues of
  third-order tensors}, Linear Algebra and its Applications, 612 (2021),
  pp.~357--374.

\bibitem{Ma2022randomized}
{\sc A.~Ma and D.~Molitor}, {\em Randomized \uppercase{K}aczmarz for tensor
  linear systems}, BIT Numerical Mathematics, 62 (2022), pp.~171--194.

\bibitem{Miao}
{\sc Y.~Miao, L.~Qi, and Y.~Wei}, {\em Generalized tensor function via the
  tensor singular value decomposition based on the \uppercase{T}-product},
  Linear Algebra and its Applications, 590 (2020), pp.~258--303.

\bibitem{Miao2021}
\leavevmode\vrule height 2pt depth -1.6pt width 23pt, {\em T-\uppercase{J}ordan
  canonical form and \uppercase{T}-drazin inverse based on the
  \uppercase{T}-product}, Communications on Applied Mathematics and
  Computation, 3 (2021), pp.~201--220.

\bibitem{Qi}
{\sc L.~Qi and X.~Zhang}, {\em T-quadratic forms and spectral analysis of
  \uppercase{T}-symmetric tensors}, arXiv preprint arXiv:2101.10820,  (2021).

\bibitem{Reichel}
{\sc L.~Reichel and U.~O. Ugwu}, {\em The tensor
  \uppercase{G}olub--\uppercase{K}ahan--\uppercase{T}ikhonov method applied to
  the solution of ill-posed problems with a \uppercase{T}-product structure},
  Numerical Linear Algebra with Applications, 29 (2022), p.~e2412.

\bibitem{Reichel2022tensor}
\leavevmode\vrule height 2pt depth -1.6pt width 23pt, {\em \uppercase{T}ensor
  \uppercase{A}rnoldi--\uppercase{T}ikhonov and \uppercase{GMRES}-type methods
  for ill-posed problems with a \uppercase{T}-product structure}, Journal of
  Scientific Computing, 90 (2022), pp.~1--39.

\bibitem{Reichel2022weighted}
\leavevmode\vrule height 2pt depth -1.6pt width 23pt, {\em Weighted tensor
  \uppercase{G}olub--\uppercase{K}ahan--\uppercase{T}ikhonov-type methods
  applied to image processing using a \uppercase{T}-product}, Journal of
  Computational and Applied Mathematics, 415 (2022), p.~114488.

\bibitem{Saad2003}
{\sc Y.~Saad}, {\em Iterative methods for sparse linear systems}, SIAM, 2003.

\bibitem{Zeng}
{\sc C.~Zeng and M.~K. Ng}, {\em Decompositions of third-order tensors:
  \uppercase{HOSVD}, \uppercase{T}-\uppercase{SVD}, and beyond}, Numerical
  Linear Algebra with Applications, 27 (2020), p.~e2290.

\bibitem{Zhang}
{\sc F.~Zhang and Q.~Zhang}, {\em Eigenvalue inequalities for matrix product},
  IEEE Transactions on Automatic Control, 51 (2006), pp.~1506--1509.

\end{thebibliography}
